\documentclass[11pt]{article}

\usepackage{amssymb}
\usepackage{mathtools}
\usepackage{graphicx}
\usepackage{amsmath}
\usepackage{verbatim}
\usepackage{setspace}
\usepackage{textpos}
\usepackage{changepage}
\usepackage{url}
\usepackage[title]{appendix}
\usepackage{color}
\usepackage{amsthm}
\usepackage[round]{natbib}

\mathtoolsset{showonlyrefs}

\newtheorem{assumption}{Assumption}
\newtheorem{proposition}{Proposition}

\newtheorem{lemma}{Lemma}

\newtheorem{assumptionalt}{Assumption}[assumption]

\newenvironment{assumptionp}[1]{
  \renewcommand\theassumptionalt{#1}
  \assumptionalt
}{\endassumptionalt}

\newcommand{\ve}{\varepsilon}

\DeclareMathOperator*{\argmax}{\arg\!\max}

\begin{document}

{
\singlespacing
\title{\textbf{A Shrinkage Likelihood Ratio Test for High-Dimensional Subgroup Analysis with a Logistic-Normal Mixture Model}
\author{Shota Takeishi\footnote{I would like to thank Katsumi Shimotsu for helpful suggestions and sharing his computational resources with me,
as well as Xuming He and Junichiro Yoshida for pointing me to the relevant literature. I also appreciate comments from participants of the weekly Bayesian seminar at the Department of Statistics, the Graduate School of Economics, the University of Tokyo.
This research was supported by JSPS KAKENHI Grant Number JP22J12024 and in part through computational resources and services provided by Advanced Research Computing at the University of Michigan, Ann Arbor.}\\
Risk Analysis Research Center \\
The Institute of Statistical Mathematics \\
shotakeishi2@gmail.com
}}
\maketitle
}

\begin{abstract} 
  In subgroup analysis, testing the existence of a subgroup with a differential treatment effect serves as protection against spurious subgroup discovery.
  Despite its importance, this hypothesis testing possesses a complicated nature: parameter characterizing subgroup classification is not identified under the null hypothesis of no subgroup. 
  Due to this irregularity, the existing methods have the following two limitations.
  First, the asymptotic null distribution of test statistics often takes an intractable form, which necessitates computationally demanding resampling methods to calculate the critical value. 
  Second, the dimension of personal attributes characterizing subgroup membership is not allowed to be of high dimension.
  To solve these two problems simultaneously, this study develops a shrinkage likelihood ratio test for the existence of a subgroup using a logistic-normal mixture model.
  The proposed test statistics are built on a modified likelihood function that shrinks possibly high-dimensional unidentified parameters toward zero under the null hypothesis while retaining power under the alternative.
  This shrinkage helps handle the irregularity and restore the simple chi-square-type asymptotics even under the high-dimensional regime.
\end{abstract}

\section{Introduction}
Subgroup analysis is routinely conducted in clinical trials that aim to account for patients' heterogeneous responses to treatment \citep{wang2007statistics}.
By exploring the interaction between treatment effect and patients' characteristics, 
subgroup analysis searches for a subgroup with certain attributes who have a more beneficial or adverse treatment effect compared to the rest of the population.
Despite its widespread usage, one possible concern is that the treatment effect is actually homogeneous so that the detected subgroup is spurious. 

To prevent such false subgroup discovery, this study proposes a new testing method for the existence of a subgroup that is computationally efficient and scales with high-dimensional patients' characteristics.
Following  \cite{shen2015inference} and \cite{shen2017penalized}, our hypothesis testing is based on a logistic-normal mixture model.
This is a type of normal mixture regression model in which the means for different Gaussian components express distinctive treatment-outcome relationships and the mixing proportion varies as a logistic function of covariates.
These covariates and their associated parameters pertain to subgroup classification and are thus called classification covariates and classification parameters, respectively.
With this model, hypothesis testing for the existence of a subgroup reduces to testing the number of components.
Specifically, the null hypothesis of one component suggests that no subgroup characterized by the classification covariates exists, and the alternative hypothesis of two components indicates otherwise.

Numerous model-based methods have been developed for testing the existence of a subgroup with a differential treatment effect.
Within the framework of the logistic-normal mixture model, a pioneering work by \cite{shen2015inference} considers an EM test for subgroup existence while \cite{shen2017penalized} extend their approach to unequal variance cases.
For survival data, \cite{wu2016subgroup} consider the logistic-Cox mixture model and develop an EM test for the existence of a subgroup.
In contrast to those mixture-based modelings, \cite{fan2017change} deal with a structurally similar change-plane model where two regression functions switch according to single-index thresholding characterized by covariates and parameters.
They then propose a score-type test for the existence of a subgroup.
This change-plane-based approach has been adapted to many other contexts: \cite{kang2017subgroup} for survival data and \cite{huang2021threshold} for binary response data.

Hypothesis testing for the existence of a subgroup possesses an irregular structure: the classification parameter is not identified under the null hypothesis of no subgroup.
Due to the presence of unidentified parameters, the aforementioned works have the following two limitations.
First, the asymptotic null distribution of test statistics takes complicated forms.
In fact, the asymptotic null distribution derived in \cite{fan2017change} is a functional of stochastic processes indexed by unidentified parameters.
The intractability of the limiting distribution necessitates computationally demanding resampling methods to calculate the critical value.
The dependence of the null distribution on the unidentified parameters further gives rise to the second limitation.
Namely, the asymptotic null distribution of test statistics is even not well defined when the dimension of classification covariates, and accordingly, the classification parameter increases with sample size.
Considering the growing availability of high-dimensional personal characteristics such as biomarkers and genetic information, this restriction can be a hurdle to the practical use of testing for the existence of a subgroup.

To address these challenges, this study develops a new testing procedure.
Our test statistics are based on the likelihood ratio as in \cite{shen2015inference} and \cite{shen2017penalized}.
However, instead of naively estimating the model parameter, we propose to estimate the parameter under the alternative model with a modified likelihood function that penalizes $L_1$-norm of the classification parameter.
When the null hypothesis is true, this penalization strongly shrinks the unidentified classification parameters toward zero.
Owing to this shrinkage, the asymptotic null distribution of the resulting shrinkage likelihood ratio test statistics ($SLRT$) 
follows the half chi-square distribution, whose quantile is easy to calculate. This asymptotic result holds even when the dimension of the classification covaraites and parameters
increases with sample size under some rate conditions. 
Meanwhile, our simulation study reveals that the proposed test is more powerful than the test that fixes the classification parameter to zero in advance; hence the shrinkage effect of the proposed test seems not as strong under the alternative hypothesis.

The penalization approach for the sake of the tractable asymptotic null distribution dates back to a seminal work by \cite{chen2001modified}, which employs penalized likelihood for testing the number of components in finite mixture models with covariate-independent mixing proportions.
This approach is further explored, for example, in \cite{li2009non} and \cite{chen2009hypothesis} for EM tests and \cite{kasahara2014modified} for Markov regime-switching models.
Although those predecessors and the present study share the same spirit of using penalization, the penalized parameters are substantially different; the former penalize 
a one-dimensional mixing proportion while the latter penalizes possibly high-dimensional parameters that can be associated with covariates.
For more general setting beyond the context of finite mixture models, \cite{fukumizu2004} and \cite{yoshida2024penalized} propose to penalize unidentified parameter for the simplified asymptotic distribution.
These works, however, do not cover the case where the unidentified parameters are high-dimensional.
Dealing with the high dimensionality requires a distinct proof strategy to establish the asymptotic result.

We note that \cite{wang2016logistic} develops a information-criterion-based method for selecting the number of components in the logistic-mixture model with high-dimensional covariates.
As \cite{chen2012inference} point out, however, such a model selection procedure and hypothesis testing often serve different purposes. 
The former is expected to find the simplest model consistent with the observed data, while the latter is useful to check the validity of scientific propositions, for example, the (non)existence of a subgroup in our context.

The rest of the paper is organized as follows. Section 2 introduces the formal model setup, the proposed testing methodology, and the notation.
Section 3 then investigates the theoretical properties of the proposed test statistics after providing the required assumptions. 
In particular, we establish the asymptotic distribution of the test statistics under the null hypothesis of no subgroup.
Section 4 illustrates the finite sample performance of the proposed method through Monte Carlo simulations. We also discuss the choice of a tuning parameter, and compare the proposed method to an EM test by \cite{shen2015inference}.
Subsequently, in Section 5, we analyze real-world data with the proposed test. 
Section 6 concludes the article. 
All the proofs of the propositions in the main text are relegated to Appendix A, while Appendix B collects the auxiliary results and their proofs.

\section{Methodology}
Let $\{(Y_i, X_i, D_i, Z_i, \varepsilon_i, \delta_i) \}_{i = 1}^n$ be $i.i.d.$ copies of sample size $n$ defined on some underlying probability space $(\Omega, \mathcal F, \mathbb P)$ with the following logistic-normal mixture model:
\begin{align}\label{model} 
  Y_i &= X_i'\alpha + D_i (\beta + \delta_i \lambda) + \ve_i, \notag \\
  \mathbb P(\delta_i = 1 | X_i, Z_i) &= \exp (Z_i'\gamma) / (1 + \exp (Z_i'\gamma)), \notag \\
  \mathbb P(\delta_i = 0 | X_i, Z_i) &=  1 - \mathbb P(\delta_i = 1 | X_i, Z_i), 
\end{align}
where $Y_i \in \mathbb R$ is the outcome of interest, $D_i \in \mathbb R$ is a treatment variable, $X_i \in \mathbb R^{q}$ is other confounding variables, and $\varepsilon_i$ is an independent error term that follows the normal distribution with mean zero and variance $\sigma^2$.
Furthermore, $\delta_i \in \{ 0, 1 \}$ is a latent subgroup membership indicator, and $Z_i \in \mathbb R^{d_n}$ is possibly high-dimensional classification covariates that may be predictive of subgroup membership.
Among the unknown parameters $(\alpha, \beta, \lambda, \gamma, \sigma^2)$, $\beta$ expresses treatment effect common to the entire population, $\lambda$ is an additional treatment effect specific to a subgroup and 
$\gamma$ is the classification parameter that governs how $Z$ influences the subgroup classification. Based on this model, we perform the following hypothesis test based on the observable $\{ (Y_i, X_i, D_i, Z_i)\}_{i = 1}^n$:
\begin{equation}\label{test}
  H_0: \lambda = 0 \ \text{against} \ H_a: \lambda > 0,
\end{equation}
where the positivity of $\lambda$ under the alternative hypothesis reflects 
an identifiability issue in the logistic-normal mixture model \citep{jiang1999identifiability}. In this formulation, the null hypothesis indicates that there exists no subgroup characterized by $Z$.

As in \cite{shen2015inference}, our likelihood ratio-based test starts with the following conditional density function of $Y_i$ given $W_i := (X_i', D_i, Z_i')'$:
\begin{equation}\label{density}
  f(Y_i|W_i; \theta, \gamma) := \pi (Z_i'\gamma) \phi_{\sigma} (Y_i - X_i'\alpha - D_i(\beta + \lambda)) + (1 - \pi(Z_i'\gamma)) \phi_{\sigma} (Y_i - X_i'\alpha - D_i\beta),
\end{equation}
where $\theta$ collects the parameter $(\alpha, \beta, \lambda, \sigma^2)$ except for $\gamma$, $\pi(x) := \exp(x) / (1 + \exp(x))$ and $\phi_{\sigma}$ is a density function of 
the normal distribution with mean 0 and variance $\sigma^2$.
Letting $l_n (\theta, \gamma) := \sum_{i = 1}^n \log f(Y_i|W_i; \theta, \gamma)$,
the likelihood ratio test statistics take the form $2 (l_n (\hat \theta, \hat\gamma) - l_n (\hat \theta_{0}))$,
where $(\hat \theta, \hat \gamma)$ is the MLE under the full logistic-normal mixture model while $\hat \theta_0$ denote the MLE under the null model with the restriction $\lambda = 0$ (so that $\gamma$ is dropped for brevity).

The standard chi-square-type limit theory is expected to break down for the likelihood ratio test in the logistic-normal mixture model.
To illustrate this point intuitively, the non-identifiability of $\gamma$ keeps $\hat \gamma$ from having any clear probability limit under the null hypothesis, $\lambda = 0$.
This non-limit property implies that $\hat \gamma$ freely moves across the whole parameter space even asymptotically, which should translate into the complex asymptotic null distribution
characterized by the parameter space of $\gamma$ as in \cite{fan2017change}.
Although this argument does not directly apply to an EM test of \cite{shen2015inference}, their test is also not free from the intractable asymptotic null distribution.
Indeed, Remark 1 of \cite{shen2015inference} suggests that, when the number of the initial values for the test is more than one, the asymptotic null distribution of the EM test is the maximum of several correlated chi-square distributions, whose quantile is not easy to calculate.
Furthermore, the adaptability of their method to high-dimensional $Z$ has not been known.

For a simple chi-square-type asymptotic distribution and high-dimensional adaptability, our proposal aims to fix the aforementioned non-limit issue of $\hat \gamma$.
To introduce our idea, we define a penalized log-likelihood function:
\begin{equation}\label{pll}
  l^{\ast}_n (\theta, \gamma) := \sum_{i = 1}^n \log f(Y_i | W_i; \theta, \gamma) - p_n \| \gamma \|_1,
\end{equation}
where $\| \gamma \|_1 := \sum_{j = 1}^d |\gamma_j|$ is the $L_1$-norm and $p_n$ is a tuning parameter such that $p_n / n$ goes to zero as $n \rightarrow \infty$.
We then use the penalized estimator $(\hat \theta^{\ast}, \hat \gamma^{\ast}) := \argmax_{\theta \in \Theta, \gamma \in \Gamma} l^{\ast}_n (\theta, \gamma)$ in place of $(\hat \theta, \hat \gamma)$ for likelihood ratio test statistics where $\Theta$ and $\Gamma$ are parameter spaces for $\theta$ and $\gamma$, respectively. 
Hence, our proposed shrinkage likelihood ratio test statistics ($SLRT$) are defined as 
\begin{equation} 
  SLRT := 2(l_n (\hat \theta^{\ast}, \hat \gamma^{\ast}) - l_n (\hat \theta_0)).
\end{equation}

The idea of penalizing $\gamma$ is inspired by the following insight. Under the null hypothesis, $\lambda = 0$, an estimate of $\lambda$ is expected to be close to zero asymptotically.
In this case, variation of $\gamma$ has little effect on $l_n (\theta, \gamma)$.
Then, the effect of $\gamma$ on $l_n^{\ast} (\theta, \gamma)$ is more through $- p_n \| \gamma \|_1$ than through $l_n (\theta, \gamma)$, which strongly shrinks $\gamma$ to zero.
This shrinkage of $\gamma$ toward zero solves the non-limit problem of $\gamma$.
In contrast, under the alternative $\lambda > 0$, an estimate of $\lambda$ should be bounded away from zero when the sample size is large.
The effect of the variation of $\gamma$ on $l_n (\theta, \gamma)$ is nontrivial this time. Combining this with the fact that $p_n / n$ is set to be asymptotically negligible,
the effect of $\gamma$ on $l^{\ast} (\theta, \gamma)$ is more through $l_n (\theta, \gamma)$ so that the shrinkage effect of $\gamma$ to zero is not as strong.
This avoids substantial power loss of the test.
Our choice of $L_1$-norm as a penalty term might appear arbitrary, but our proof of the asymptotic results depends crucially on the form of the $L_1$-norm, as discussed in the next section.

In the remainder of the paper, we use the following notation. Collect the covariate as $U_i = (X_i', D_i)'$. 
We suppress the dependence of $d_n$ and $p_n$ on $n$ and just write $d$ and $p$.
Let $:=$ denote ``equals by definition.''
For $a \in \mathbb R^k$, $\| a \|_r$ denotes the $L_r$-norm of $a$ in Euclidean space. 
In particular, we suppress $r$ and just write $\| a \|$ when referring to the $L_2$-norm.
For $a := (a_1, \dots, a_k)' \in \mathbb R^{k}$ and a real-valued function $g(a)$, let $\nabla_{a} g(a^{\ast}) := (\partial g(a^{\ast}) / \partial a_1, \dots, \partial g(a^{\ast}) / \partial a_2)'$ be a vector of derivative evaluated at $a = a^{\ast}$.
The subscript $0$ as in $\theta_0$ signifies the true parameter value. 
For two real numbers $a$ and $b$, $a \land b$ and $a \lor b$ denote $\min(a, b)$ and $\max(a, b)$, respectively.
Let $\mathcal C$ be a universal finite positive constant whose value may change from one expression to another.
For two real sequences $\{ a_n \}_{n \in \mathbb N}$ and $\{ b_n \}_{n \in \mathbb N}$, the notation $a_n \lesssim b_n$ means that there exists a finite constant $\mathcal D$ independent of $n$ such that $a_n \leq \mathcal D b_n$ for all $n \in \mathbb N$.
All the limits below are taken as $n \rightarrow \infty$ unless stated otherwise.
Throughout the article, we assume $n \land d \geq 2$.

\section{Theory}
The goal of this section is to establish the asymptotic null distribution of $SLRT$, which is crucial for the implementation of the test.
After setting the required assumptions, we first show the consistency of $\hat \theta^*$ and the convergence rate of $\| \hat \gamma^* \|_1$ to zero (Proposition \ref{consistency}).
We then introduce the reparameterization of $\theta$ to deal with the non-regularity inherent in the logistic-normal mixture model.
Built on the reparameterization, we establish the quadratic approximation for the penalized log-likelihood function and derive the convergence rate of the reparameterized estimator and the asymptotic null distribution of $SLRT$.
In the following, let $\Theta := \Theta^{\alpha} \times \Theta^{\beta} \times \Theta^{\lambda} \times \Theta^{\sigma^2}$ be the parameter space for $\theta = (\alpha', \beta, \lambda, \sigma^2)'$,
and $Z := (Z_{(1)}, \cdots, Z_{(d)})'$. Throughout the section, we assume that the null hypothesis holds: $\lambda_0 = 0$.
\subsection{Assumptions}
In addition to the basic model setup, the following set of assumptions is required for the subsequent theoretical results.
\begin{assumption}\label{parameter}
  (a) $\Theta^{\alpha}$ and $\Theta^{\beta}$ are compact, convex sets,
  (b) $\Theta^{\sigma^2} = (0, \infty)$, (c) $\Theta^{\lambda} = [0, u_{\lambda}]$ for some $0 < u_{\lambda} < \infty$, (d) $\Gamma = \mathbb R^d$ and (e) $(\alpha_0', \beta_0)'$ lies in an interior of $\Theta^{\alpha} \times \Theta^{\beta}$.
\end{assumption}

\begin{assumption} \label{covariate}
  (a) $\mathbb E [\| U \|^{10}] < \infty$,
  (b) $\mathbb E[U U']$ is positive definite,
  (c) $Z$ is uniformly sub-Gaussian: there exist finite $K, C > 0$, independent of $n$ and $j$, such that $\mathbb P (|Z_{(j)}| > t) \leq K e^{-C t^2}$ for all $0 < t < \infty$ and $1 \leq j \leq d$, and
  (d) $D$ is bounded and nondegenerate, i.e., $Var(D) > 0$.
\end{assumption}

\begin{assumption}\label{tuning parameter}
  $n$, $d$ and $p$ satisfy the following rate condition: $(a)$ $n^{7/4} \sqrt{\log n} \log d / p^2 = o(1)$ and $(b)$ $\log d = o(n^{1/4})$.
\end{assumption}
Compactness in Assumption \ref{parameter}$(a)$ and $(c)$ is required for the proof of consistency (Proposition \ref{consistency}), which extends that of Lemma A1 of \cite{andrews1993tests} to our context.
As noted by Assumption \ref{parameter}$(b)$ and $(d)$, the parameter spaces for $\sigma^2$ and $\gamma$ are unrestricted.
Still, Lemma \ref{compact} suggests that those parameter spaces can essentially be regarded as compact with probability approaching one.
Assumption \ref{covariate}$(a)$ is set because we expand the log-likelihood function five times and the tenth-order terms of $U$ appear when establishing the quadratic approximation for the penalized log-likelihood function (Proposition \ref{quad approx}).
Similar higher-order moment conditions are employed when higher-order expansion of the log-likelihood function is necessary, as in Assumption 2$(a)$ of \cite{kasahara2015testing}.
Positive definiteness of $\mathbb E[U U']$ in Assumption \ref{covariate}$(b)$ is standard in hypothesis testing for finite mixture models and can be seen in Theorem 2 of \cite{shen2015inference} and Assumption 2$(b)$ of \cite{kasahara2015testing}.
We, however, do not require positive definiteness of $\mathbb E[Z Z']$, which is a major departure from \cite{shen2015inference} (see Theorem 2 therein).
Sub-Gaussianity in Assumption \ref{covariate}$(c)$ is for the sake of repeated use of Lemma 2.2.1 and 2.2.2 of \cite{van1996weak} to deal with the high-dimensionality of $Z$.
For Assumption \ref{covariate}$(d)$, the boundedness is a technical requirement for the proof of Lemma \ref{shrinkage pi} while the nondegeneracy avoids the complication associated with the quadratic approximation, as discussed in the paragraph following Proposition \ref{quad approx}.
This assumption should be satisfied in most of our intended applications where $D$ denotes a treatment variable in clinical trials.
Note that \cite{fan2017change} also consider bounded and nondegenerate $D$ in their setting for subgroup analysis.
Assumption \ref{tuning parameter}$(b)$ indicates that $d$ can grow much faster than $n$.
However, given Assumption \ref{tuning parameter}$(a)$, larger $d$ requires larger $p$, which leads to stronger shrinkage and thus power loss of the test.
We also note that, in contrast to $d$ for the classification covariates, the dimension of $X_i$ in the linear regression is required to be fixed. This restriction might appear unnatural. However, high-dimensionality of the identified parameter $\alpha$ in the non-convex likelihood function would pose yet another intractable challenge as documented in, for example, \cite{stadler2010test}, and, thus, is beyond the scope of the present study.

\subsection{Asymptotic Null Distribution of $SLRT$}
Based on the assumptions, we start by showing the consistency of $\hat \theta^*$ and the convergence rate of $\| \hat \gamma^* \|_{1}$ to zero.

\begin{proposition} \label{consistency}
  Assume Assumptions \ref{parameter}-\ref{tuning parameter} hold.
  Then $\hat \theta^* \rightarrow_p \theta_0$ and $\| \hat \gamma^* \|_1 = o_p (n^{-1/4} (\log d \log n)^{-1/2})$.
\end{proposition}
Note that the choice of the convergence rate of $\| \hat \gamma^* \|_1$, $n^{-1/4} (\log d \log n)^{-1/2}$, is for the sake of the proof of Proposition \ref{quad approx} and not essential in itself.
The proof is built on that of Lemma A1 of \cite{andrews1993tests}, a consistency result when some parameters are not identified.
We make modifications so that the unidentified parameter can be of high-dimension and its $L_1$ norm converges to zero in probability.
The key instrument for handling the high-dimensionality is a judicious use of the multivariate contraction principle (Lemma \ref{multivariate contraction}) that leverages the contraction property of the multivariate function $(x_1, x_2, x_3) \rightarrow \log (\pi (x_1) e^{-x_2} + (1 - \pi (x_1)) e^{-x_3})$ appearing in the log-density.
As indicated in the inequality \eqref{prop consistency multivariate} and Lemma \ref{maximal}, bounding the $L_1$ norm of $\gamma$ is vital in proving the convergence $\hat \theta^* \rightarrow_p \theta_0$ under the rate condition given in Assumption \ref{tuning parameter} in our proof.
This bound is, in turn, possible through the $L_1$ penalty as shown in Lemma \ref{compact}; hence, the choice of $L_1$-norm as a penalty is essential in our proof strategy.

We proceed to analyze the asymptotic properties of $SLRT$.
As $\lambda_0 = 0$ is on the boundary of $\Theta^{\lambda}$ under the null hypothesis,
we employ the method of \cite{andrews1999estimation} for quadratic approximation of the penalized log-likelihood function with the score $\nabla_{\theta} \log f(Y | W; \theta_0, \hat \gamma^*)$.
We, however, note that the standard analysis is hampered by the irregular structure of the score:
\begin{align} \label{original score}
  &\frac{\partial}{\partial \beta} \log f(Y|W; \theta_0, \hat \gamma^*) = \frac{D (Y - X'\alpha_0 - D \beta_0)}{\sigma^2_0} \notag, \\
  &\frac{\partial}{\partial \lambda} \log f(Y|W; \theta_0, \hat \gamma^*) = \pi (Z'\hat \gamma^*) \frac{D (Y - X'\alpha_0 - D\beta_0)}{\sigma^2_0}.
\end{align}
There are two aspects to this irregularity.
First, $(\partial / \partial \lambda) \log f(Y | W; \theta_0, \hat \gamma^*)$ depends on $\hat \gamma^*$, the dimension of which possibly increases with the sample size $n$.
However, this problem can be solved by using the convergence $\| \hat \gamma^* \|_1 = o_p (n^{-1/4} (\log d \log n)^{-1/2})$ in Proposition \ref{consistency}.
Specifically, Lemma \ref{shrinkage pi} clarifies that $\pi (Z'\hat \gamma^*)$ can be approximated by $\pi (0)$ asymptotically, which enables us to treat $(\partial / \partial \lambda) \log f(Y | W; \theta_0, \hat \gamma^*)$ essentially as $D(Y - X'\alpha_0 - D\beta_0) / 2 \sigma^2_0$.
Unfortunately, this approximation leads to the second aspect of the irregularity: $(\partial / \partial \beta) \log f(Y | W; \theta_0, \hat \gamma^*)$ and $D(Y - X'\alpha_0 - D\beta_0) / 2 \sigma^2_0$ are linearly dependent.
This linear dependence degenerates the Fisher information matrix so that the standard second-order quadratic approximation is no longer valid.

We overcome this challenge by considering reparameterization inspired by \cite{kasahara2015testing}, which are built on the result of \cite{rotnitzky2000likelihood}.
Let us introduce the following one-to-one mapping between the original parameter $(\alpha', \beta, \sigma^2, \lambda)'$ and the reparameterized one $(\alpha', \nu, \sigma^2, \lambda)'$ with $\beta = \nu - \lambda / 2$.
Collect the reparameterized parameter as $\psi := (\eta', \sigma^2, \lambda)'$ where $\eta := (\alpha', \nu)'$.
Accordingly, let $\Theta^{\psi} := \{ (\alpha', \beta + \lambda/2 , \sigma^2, \lambda)': (\alpha', \beta, \lambda, \sigma^2) \in \Theta \}$ denote the reparameterized parameter space.
With the abuse of notation, the reparameterized density, the log-likelihood function, and the penalized log-likelihood function are given by 
\[
  f(Y|W; \psi, \gamma) := \pi (Z'\gamma) \phi_{\sigma} (Y - X'\alpha - D(\nu + \lambda/2)) + (1 - \pi (Z'\gamma)) \phi_{\sigma} (Y - X'\alpha - D(\nu - \lambda/2)),
  \]
$l_n (\psi, \gamma) := \sum_{i = 1}^n \log f(Y_i | W_i; \psi, \gamma)$ and $l_n^{\ast} (\psi, \gamma) := l_n (\psi, \gamma) - p \| \gamma \|_1$.
With this reparameterization, the original score structure \eqref{original score} can be transformed into
\begin{align} 
  \frac{\partial}{\partial \nu} \log f(Y | W; \psi_0, \hat \gamma^*) &= \frac{D (Y - X'\alpha_0 - D\nu_0)}{\sigma^2_0} \notag \\
  \frac{\partial}{\partial \lambda} \log f(Y | W; \psi_0, \hat \gamma^*) &= (2 \pi (Z'\hat \gamma^*) - 1) \frac{D(Y - X'\alpha_0 - D\nu_0)}{2 \sigma_0^2} \notag \\
  \frac{\partial^2}{\partial \lambda^2} \log f(Y | W; \psi_0, \hat \gamma^*) &= \frac{D^2}{4 \sigma^2_0} \left\{ \frac{(Y - X'\alpha_0 - D \nu_0)^2}{\sigma_0^2} - 1 \right\} - \left(\frac{\partial}{\partial \lambda} \log f(Y | W; \psi_0, \hat \gamma^*)   \right)^2. \notag
\intertext{As suggested by Lemma \ref{shrinkage pi}, the following approximation holds for the second and the third line:}
  \frac{\partial}{\partial \lambda} \log f(Y | W; \psi_0, \hat \gamma^*) &\approx 0, \notag \\
  \frac{\partial^2}{\partial \lambda^2} \log f(Y | W; \psi_0, \hat \gamma^*) &\approx \frac{D^2}{4 \sigma^2_0} \left\{ \frac{(Y - X'\alpha_0 - D \nu_0)^2}{\sigma_0^2} - 1 \right\}. \label{score approx}
\end{align}
As $(\partial / \partial \nu) \log f(Y | W; \psi_0, \hat \gamma^*)$ and the approximation for $(\partial / \partial \lambda^2) \log f(Y | W; \psi_0, \hat \gamma^*)$ in \eqref{score approx} are linearly independent, the derivative of the penalized log-likelihood function with respect to $\alpha$, $\nu$, $\sigma^2$ and $\lambda^2$ can play the role of the scores for the quadratic approximation under this parameterization. 
Namely, let 
\begin{equation} 
  t_n (\psi) :=
    \begin{pmatrix} 
      n^{1/2} (\eta - \eta_0) \\
      n^{1/2} (\sigma^2 - \sigma^2_0) \\
      n^{1/2} \lambda^2
    \end{pmatrix},
    s_i :=
    \begin{pmatrix} 
      \frac{U_i}{\sigma_0} H^1_i \\
      \frac{1}{2 \sigma^2_0} H^2_i \\
      \frac{D_i}{8 \sigma^2_0} H^2_i
    \end{pmatrix}, \notag
\end{equation}
$S_n := n^{-1/2} \sum_{i = 1}^n s_i$, $\mathcal I_n := n^{-1} \sum_{i = 1}^n s_i s_i'$
and $H^k_i := H^k (\ve_i / \sigma_0)$ for $k \in \mathbb N$ where $H^k(z)$ is the Hermite polynomial of order $k$ given by, for example, $H^1 (z) = z$ and $H^2 (z) = z^2 - 1$.
Then, the following proposition formally establishes the quadratic approximation for the penalized log-likelihood function.
\begin{proposition} \label{quad approx}
  Assume Assumptions \ref{parameter}-\ref{tuning parameter} hold. Then
  (a) $l^{\ast}_n (\psi, \hat \gamma^*) - l^{\ast}_n (\psi_0, \hat \gamma^*) = S_n' t_n (\psi) - \frac{1}{2} t_n (\psi)' \mathcal I_n t_n (\psi) + R_n (\psi, \hat \gamma^*)$,
  where $\sup_{\psi \in \{ \psi \in \Theta^{\psi} : \| \psi - \psi_0 \| \leq \kappa \}} |R_n (\psi, \hat \gamma^*)| / (1 + \| t_n (\psi) \|)^2 \rightarrow_p 0$ for any sequence $\kappa$ converging to zero,
  (b) $S_n \rightarrow_d N(0, \mathcal I)$, and (c) $\mathcal I_n \rightarrow_p \mathcal I$, where $\mathcal I$ is nonsingular.
\end{proposition}
Proposition 2 of \cite{kasahara2015testing} obtains a similar quadratic approximation result under reparameterization for their analysis of normal mixture regression models.
Our proposition \ref{quad approx}, however, departs from that work in two respects.
First, the setting of \cite{kasahara2015testing} accommodates heterogeneous intercept terms across different mixture components,
which corresponds to the case where $D$ equals unity in our context.
Such heterogeneity is yet another source of singularity of the Fisher information matrix: the first and second derivatives of the log-density with respect to variance and an intercept term, respectively, are linearly dependent.
Due to this  complication, \cite{kasahara2015testing} employ a more involved reparameterization than ours.
Second, our proof faces a new challenge of the presence of possibly high-dimensional $\hat \gamma^*$ and handle the problem in a novel approach.
Namely, we show that the effect of $\hat \gamma^*$ on the quadratic approximation vanishes asymptotically (Lemma \ref{shrinkage pi}) with the help of Proposition \ref{consistency}.
Consequently, the resulting quadratic approximation in Proposition \ref{quad approx}$(a)$ is the same as if $\hat \gamma^*$ is fixed to zero up to the remainder term.

Define the reparameterized version of the penalized MLE as $(\hat \psi^*, \hat \gamma^*) := \argmax_{\psi \in \Theta^{\psi}, \gamma \in \Gamma} l^{\ast}_n (\psi, \gamma)$.
Then, based on Proposition \ref{quad approx}, the following proposition derives the asymptotic null distribution of $SLRT$.
Note that $\chi_1^2/2 + \chi_0^2/2$ denotes the half chi-square distribution, which is a mixture of the chi-square distribution with one degree of freedom and a point mass at zero with equal mixing weights.
\begin{proposition} \label{asy null dist}
  Assume Assumptions \ref{parameter}-\ref{tuning parameter} hold.
  Then (a) $t_n (\hat \psi^*) = O_p (1)$ and $(b)$ $SLRT \rightarrow_d \chi_1^2/2 + \chi_0^2/2$.
\end{proposition}
According to the proof, the limiting null distribution, $\chi_1^2/2 + \chi_0^2/2$, is the same as that for the likelihood ratio test statistics with $\gamma$ fixed to zero.
Moreover, the difference between the latter test statistics and $SLRT$ converges to zero in probability.
The half chi-square distribution often appears in likelihood-ratio-based tests for the number of components in finite mixture models in which penalization is imposed on one-dimensional mixing proportion independent of covariate \citep[e.g.,][]{chen2001modified,li2009non}.
In this respect, our result witnesses the generalizability of the penalization approach in the literature to covariate-dependent and high-dimensional settings.

\section{Monte Carlo Simulation}
In this section, we first discuss the criteria for the choice of tuning parameter $p$. 
We then conduct a simulation study to evaluate the finite sample performance of the test under several settings.
Finally, we compare the proposed method to an EM test by \cite{shen2015inference}, which, while not intended for the high-dimensional setting, is also based on the logistic-normal mixture model.
We compute $(\hat \theta^*, \hat \gamma^*)$ via a version of EM algorithm \citep{jordan1994hierarchical} where the standard weighted logistic regression in the M step is replaced by its $L_1$-penalized counterpart.
The monotonicity of the algorithm can be easily shown.
All the simulations are conducted in R language \citep{rlanguage}.
The simulation codes are available at the author's GitHub repository (\url{https://github.com/stakeish/subgroupSLRT}).
\subsection{Empirical Formula for $p$}
While Assumption \ref{tuning parameter} indicates the order of $p$ relative to $n$ and $d$, this assumption per se does not provide the exact value of $p$ the practitioners should choose. 
Inspired by \cite{chen2011tuning}, we derive the empirical formula determining the specific value of $p$ given $n$ and $d$ through numerical experiments.
The construction of the formula proceeds as follows.
For each $n \in \{ 100, 250, 500, 750, 1000 \}$ and $d \in \{10, 25, 50, 75, 100 \}$, we generate 2000 datasets of $i.i.d.$ random variables $\{ Y_i, X_i, D_i, Z_i \}_{i = 1}^n$, with a random seed set to 10, from the following null distribution:
$X \sim N(0, 1)$, $D \sim Bernoulli(0.5)$,  $Z_{(1)} = 1$,  $(Z_{(2)}, \dots, Z_{(d)})' \sim N(0, I)$,  $Y \sim N(1 + 2X + D, 1)$,
where $I$ is an identity matrix and $X$, $D$ and $Z$ are independent of each other.
Let $\mathcal P$ be a candidate set for $p$.
For each $n$, we first calculate the rejection frequency of the likelihood ratio test with $\gamma$ fixed to zero (the benchmark rejection frequency); then, for each $d$ and for all $p$ in $\mathcal P$, we compute rejection frequencies of the shrinkage likelihood ratio tests.
We set the level of the tests to $5 \%$ and use $\chi^2_1 / 2 + \chi^2_0 / 2$ to calculate the critical value.
Now, for each $n$ and $d$, we define $p_{n, d}$ as the smallest value in $\mathcal P$ with which the rejection frequency of the shrinkage likelihood ratio test falls within 0.3$\%$ from the benchmark rejection frequency.
With 25 observations of $(p_{n, d}, n, d)$ at hand, consider the regression model: $p_{n, d} = a + b n^{7/8} \sqrt{\log d}$, motivated by Assumption \ref{tuning parameter}$(a)$.
We obtain an estimate of $(a, b)'$ simply by the method of least squares, which results in the following empirical formula for $p$.
\begin{align} \label{empirical formula} 
  p = 6.3383 + 0.0086 n^{7/8} \sqrt{\log d}
\end{align}
We use the likelihood ratio test statistics with $\gamma$ fixed to zero as the benchmark because $SLRT$ approach these statistics asymptotically under the null as the argument following Proposition \ref{asy null dist} discusses.
Even though we set $Z$ to the standard normal variable, the empirical formula is robust to deviation from this specific distribution, as clarified later in this section.

\subsection{Size and Power}
We move on to assess the finite sample properties of the proposed methods.
A random seed is set to 20 for each $n$, $d$, and parameter value.
We use the empirical formula \eqref{empirical formula} to determine $p$.
The level is set to $5 \%$ throughout. 
We first examine the size property of the test. Consider the four null settings: $X \sim N(0, 1)$,  $D \sim Bernoulli(0.5)$,  $Z_{(1)} = 1$, $Y \sim N(1 + 2X + D, 1)$ and
\begin{align} \notag
  &\text{Setting I:} \ (Z_{(2)}, \dots, Z_{(d)})' \sim N(0, I), \ & &\text{Setting II:} \ (Z_{(2)}, \dots, Z_{(d)})' \sim N(0, \Sigma) \notag \\
  &\text{Setting III:} \ (Z_{(2)}, \dots, Z_{(d)})' \sim i.i.d. \ \text{Rademacher}, & &\text{Setting IV:} \ (Z_{(2)}, \dots, Z_{(d)})' \sim i.i.d. \ \text{Skew normal}, \notag 
\end{align}
where $\Sigma$ in Setting II is a randomly generated positive definite matrix based on Cholesky decomposition.
Furthermore, ``Rademacher'' in Setting III indicates a Rademacher variable, which takes its value at $-1$ or $1$ with equal probability 
and ``Skew normal'' in Setting IV means the skew normal distribution with the shape parameter set to 4 and the other parameters specified to ensure that the distribution has mean zero and variance one.  
We benchmark the shrinkage likelihood ratio tests in these four settings against the likelihood ratio tests with $\gamma$ fixed to zero.
Table \ref{type i error} shows that the proposed method works reasonably well in preserving the nominal level, especially when the sample size is more than 500.
The proposed method also generally attains type I errors close to those of the benchmark. 
Notably, the deviation of $Z$ from the standard normality has little effect on the performance, which can be supporting evidence for the generalizability of the empirical formula \eqref{empirical formula}.
\begin{table}[h]\caption{Type I Error} \label{type i error} \medskip 
	\centering 
	\begin{tabular}{c|cc|cc|cc|cc}
		\hline
		\multicolumn{1}{c|}{} & \multicolumn{8}{c}{$d = 10$} \\
		\hline \hline
		& B(I) & S(I) & B(II)  & S(II) & B(III) & S(III) & B(IV) & S(IV) \\
		\hline
		$n = 100$ & 0.0640 & 0.0646 & 0.0640 & 0.0646 & 0.0578 & 0.0582 & 0.0564 & 0.0572  \\
		$n = 250$ & 0.0528& 0.0552 & 0.0528 & 0.0562 & 0.0584 & 0.0596 & 0.0574 & 0.0588 \\
		$n = 500$  & 0.0520 & 0.0562 & 0.0520 & 0.0554 & 0.0482 & 0.0506 & 0.0538 & 0.0568 \\
    $n= 750$ & 0.0488 & 0.0518 & 0.0488 & 0.0512 & 0.0464 &  0.0488 & 0.0530 & 0.0570 \\
    $n = 1000$ & 0.0548 &  0.0572& 0.0548 & 0.0562 & 0.0590 & 0.0610 &  0.0460 & 0.0484\\
		\hline
    \multicolumn{1}{c|}{} & \multicolumn{8}{c}{$d = 50$} \\
    \hline \hline
		& B(I) & S(I) & B(II)  & S(II) & B(III) & S(III) & B(IV) & S(IV) \\
		\hline
		$n = 100$ & 0.0664 & 0.0688 & 0.0664 &  0.0688 & 0.0584 & 0.0600 & 0.0626 & 0.0648 \\
		$n = 250$ & 0.0582 & 0.0650 & 0.0566  & 0.0626 & 0.0588 & 0.0646 & 0.0536 & 0.0620  \\
		$n = 500$ & 0.0506 & 0.0570 &  0.0528 &0.0594 & 0.0522 & 0.0580 & 0.0500 & 0.0576 \\
    $n= 750$ & 0.0546 & 0.0578 & 0.0494 & 0.0552 & 0.0492 & 0.0524 & 0.0488 & 0.0534\\
    $n = 1000$ & 0.0480 & 0.0498 & 0.0516 &0.0532 & 0.0504 & 0.0524 & 0.0556 & 0.0584\\
		\hline
    \multicolumn{1}{c|}{} & \multicolumn{8}{c}{$d = 100$} \\
    \hline \hline
		& B(I) & S(I) & B(II)  & S(II) & B(III) & S(III) & B(IV) & S(IV) \\
		\hline
		$n = 100$ & 0.0570 & 0.0614 & 0.0620& 0.0654& 0.0598  & 0.0630  & 0.0664 & 0.0700  \\
		$n = 250$ & 0.0548 & 0.0660 & 0.0590 &0.0680& 0.0564 & 0.0670  & 0.0492 & 0.0572 \\
		$n = 500$ & 0.0522 & 0.0594 & 0.0588  & 0.0672 & 0.0522& 0.0606 & 0.0552 & 0.0640\\
    $n= 750$ & 0.0518&0.0562   &0.0476  & 0.0518 & 0.0516 & 0.0554& 0.0502 & 0.0544 \\
    $n = 1000$ & 0.0468 & 0.0494 &0.0506 &0.0526 & 0.0474& 0.0500 & 0.0478 & 0.0514  \\
		\hline
	\end{tabular} \\
	\begin{flushleft}
	\small
	Notes: The columns ``B(I)", ``B(II)'', ``B(III)" and ``B(IV)" report the rejection frequencies of the likelihood ratio test with $\gamma$ fixed to zero (benchmark), for Setting I, II, III and IV, respectively.
  The columns ``S(I)", ``S(II)'', ``S(III)" and ``S(IV)" report the rejection frequencies of the shrinkage likelihood ratio tests for Setting I, II, III and IV, respectively. The number of replications is 5000.
	\end{flushleft}
\end{table}

We then investigate the power properties of the proposed tests.
Consider the same setting as the null case except for $Y \sim N(1 + 2 X + (1 + \delta)D, 1)$ and $\mathbb P(\delta = 1| X, Z) = \pi(Z'\gamma)$, where $\gamma$ is a vector with all the elements equal 1.
Similarly to the null case, we benchmark the proposed method against the likelihood ratio test with $\gamma$ fixed to zero.
For fair comparison, we use the size-adjusted critical values, obtained from the simulation under the null.
As table \ref{power} indicates, the proposed methods generally improve upon the power over the benchmark. 
Although the degree of the improvement is marginal in not a few settings, the proposed method improves the benchmark by more than 10\% in some cases..
We note that the degree of improvement decreases with $d$, which coincides with the observation regarding Assumption \ref{tuning parameter}. 
Interestingly, the correlation of $Z$ seems to boost the power improvement.
This phenomenon suggests the possibility that the power of the test depends on the correlation structure for $Z$.
\begin{table}[h]\caption{Size-adjusted Power} \label{power} \medskip 
	\centering 
	\begin{tabular}{c|cc|cc|cc|cc}
		\hline
		\multicolumn{1}{c|}{} & \multicolumn{8}{c}{$d = 10$} \\
		\hline \hline
		& B(I) & S(I) & B(II)  & S(II) & B(III) & S(III) & B(IV) & S(IV) \\
		\hline
		$n = 100$ & 0.1712 & 0.1768 & 0.1694 & 0.1772 & 0.1832 & 0.1876 & 0.1890 & 0.1926 \\
		$n = 250$ & 0.3242 & 0.3692 & 0.3112 & 0.3952 & 0.3214 & 0.3774 & 0.3270 & 0.3714 \\
		$n = 500$ &  0.5140 & 0.6538 & 0.5016 & 0.7294& 0.5306 &  0.6660 & 0.5112 & 0.6462 \\
    $n= 750$ & 0.6730 & 0.8248 & 0.6210 & 0.8500 & 0.6858 & 0.8318 & 0.6678 & 0.8274 \\
    $n = 1000$ & 0.7750 & 0.9240 & 0.7610 & 0.9738 & 0.7548 & 0.9270 & 0.7852 & 0.9188\\
		\hline
    \multicolumn{1}{c|}{} & \multicolumn{8}{c}{$d = 50$} \\
    \hline \hline
		& B(I) & S(I) & B(II)  & S(II) & B(III) & S(III) & B(IV) & S(IV) \\
		\hline
		$n = 100$ &0.1742    &     0.1736 &0.1794     &    0.1830 & 0.1910    &     0.1958 & 0.1714    &     0.1748 \\
		$n = 250$ &0.3278    &     0.3462 &0.3384     &    0.4078 & 0.3228    &     0.3566 & 0.3418    &     0.3598 \\
		$n = 500$ &0.5560    &     0.6104 & 0.5308     &    0.6864 & 0.5330   &      0.5954 & 0.5534   &      0.6000 \\
    $n= 750$ &0.6790     &    0.7386 &0.7082     &    0.8548 & 0.7052     &    0.7604 &  0.7022    &     0.7502 \\
    $n = 1000$ &0.8134   &      0.8490 &0.8016   &      0.9302 & 0.8048   &      0.8428 & 0.7878   &      0.8372 \\
		\hline
    \multicolumn{1}{c|}{} & \multicolumn{8}{c}{$d = 100$} \\
    \hline \hline
		& B(I) & S(I) & B(II)  & S(II) & B(III) & S(III) & B(IV) & S(IV) \\
		\hline
		$n = 100$ &0.2006   &      0.2008 &0.1862   &      0.1912 & 0.1896    &     0.1892 & 0.1832    &     0.1852 \\
		$n = 250$ &0.3540   &      0.3570 &0.3344   &      0.3936 &0.3476     &    0.3378 & 0.3572     &    0.3788\\
		$n = 500$ &0.5476   &      0.5896 &0.5266    &     0.6458 & 0.5490    &     0.5842 & 0.5424     &    0.5780 \\
    $n= 750$ &0.7064    &     0.7302 &0.7138     &    0.8194 & 0.7010    &    0.7332 & 0.7138    &     0.7478 \\
    $n = 1000$ &0.8204   &      0.8364 &0.8052    &     0.8968 & 0.8178   &      0.8344 & 0.8130   &      0.8246 \\
		\hline
	\end{tabular} \\
	\begin{flushleft}
	\small
	Notes: The columns ``B(I)", ``B(II)'', ``B(III)" and ``B(IV)" report the rejection frequencies of the likelihood ratio test with $\gamma$ fixed to zero (benchmark), for Setting I, II, III and IV, respectively.
  The columns ``S(I)", ``S(II)'', ``S(III)" and ``S(IV)" report the rejection frequencies of the shrinkage likelihood ratio tests for Setting I, II, III and IV, respectively. The number of replications is 5000.
	\end{flushleft}
\end{table}

\subsection{Comparison with an EM test of \cite{shen2015inference}}
This subsection provides a brief comparison of the proposed shrinkage likelihood ratio test with an EM test proposed by \cite{shen2015inference}.
The implementation detail for the EM test is summarized as follows.
Pick $J$ initial values of $\gamma$: $\gamma_1^{(0)}, \dots \gamma_J^{(0)}$.
For each $j = 1, \dots, J$, we compute the maximum likelihood estimate $\theta^{(0)}_{j}$ based on the model \eqref{model} with $\gamma = \gamma^{(0)}_{j}$ fixed. 
We then update $(\theta^{(0)}_j, \gamma^{(0)}_j)$ via $K$-step EM iterations. Namely, for each $k = 0, \dots, K-1$, with an input $(\theta^{(k)}_j, \gamma^{(k)}_j)$, we first perform an $E$ step to compute 
\begin{equation} 
  a_i^{(k)} := \mathbb P(\delta_i = 1 | Y_i, X_i, Z_i; \theta^{(k)}_j, \gamma^{(k)}_j), i = 1, \dots, n,
\end{equation}
where $\mathbb P(\delta_i = 1 | Y_i, X_i, Z_i; \theta^{(k)}_j, \gamma^{(k)}_j)$ is the conditional probability of $\delta_i = 1$ given $Y_i, X_i, Z_i$ under the model \eqref{model} with $(\theta, \gamma) = (\theta^{(k)}_j, \gamma^{(k)}_j)$.
Then we update
\begin{align} 
  \hat \theta^{(k + 1)}_j &:= \argmax_{\theta \in \Theta} \left(\sum_{i = 1}^n [a_i^{(k)} \log \phi_{\sigma} (Y_i - X_i'\alpha - D_i (\beta + \lambda)) + (1 - a_i^{(k)}) \log \phi_{\sigma} (Y_i - X_i'\alpha - D_i \beta)] \right), \notag \\
  \hat \gamma^{(k + 1)}_j &:= \argmax_{\gamma \in \Gamma} \left(\sum_{i = 1}^n [a_i^{(k)} \log \pi (Z_i'\gamma) + (1 - a_i^{(k)}) (1 - \pi(Z_i' \gamma))] \right).
\end{align}
Obtaining $(\hat \theta_j^{(K)}, \hat \gamma_j^{(K)})$ after the $K$-step EM iterations, we compute the maximum likelihood estimate $\hat \theta_j$ based on the model \eqref{model} with $\gamma = \hat \gamma^{(K)}_j$ fixed.
Finally, we define the EM test statistics as
\begin{equation} 
  EM^{(K)} := \max_{j = 1, \dots, J} \left\{ 2\left( l_n (\hat \theta_j, \hat \gamma^{(K)}_j) - l_n (\hat \theta_0) \right)\right\},
\end{equation}
where recall that $\hat \theta_0$ is the MLE under the null model. 
As the asymptotic null distribution of $EM^{(K)}$ is not tractable when $J$ is more than one, we use bootstrap to obtain critical values as suggested by \cite{shen2015inference}.

We first compare the size properties of the tests.
The distribution of the data is the same as the null setting in subsection 4.2 with $(Z_{(2)}, \dots, Z_{(d)})' \sim N(0, I)$.
The level is set to $0.05$.
The implementation procedure for the shrinkage likelihood ratio is the same as before.
For the EM test, we consider two choices of $(J, K)$: $(5, 3)$ and $(10, 6)$, and pick initial values of $\gamma$ randomly.
The bootstrap sample size is set to $500$. Due to computational burden involved in the bootstrap, the number of replications is $2000$ for the EM test.
Table \ref{comparison type i error} summarizes the result. Although both of the two tests are generally good at the size control, the performance of the EM test is superior in that this test works well even in such small sample sizes as $n = 100$ and $n = 250$.
Interestingly enough, the EM test is not vulnerable to the high-dimensionality in terms of the size control.

\begin{table}[h]\caption{Comparison of Type I Error} \label{comparison type i error} \medskip 
	\centering 
	\begin{tabular}{c|c|cc}
		\hline
		\multicolumn{1}{c|}{} & \multicolumn{3}{c}{$d = 10$} \\
		\hline \hline
		 & $SLRT$ & $EM (5, 3)$ & $EM (10, 6)$ \\
		\hline
		$n = 100$ & 0.0646 & 0.0515 & 0.0540   \\
		$n = 250$ & 0.0552& 0.0570 & 0.0575  \\
		$n = 500$  & 0.0562 & 0.0470 & 0.0460 \\
    $n= 750$ & 0.0518 & 0.0470 & 0.0510  \\
    $n = 1000$ & 0.0572 &  0.0425 & 0.0445 \\
		\hline
    \multicolumn{1}{c|}{} & \multicolumn{3}{c}{$d = 50$} \\
		\hline \hline
		 & $SLRT$ & $EM (5, 3)$ & $EM (10, 6)$ \\
		\hline
		$n = 100$ & 0.0688 & 0.0500 & 0.0485   \\
		$n = 250$ & 0.0650& 0.0470 & 0.0525  \\
		$n = 500$  & 0.0570 & 0.0555 & 0.0465 \\
    $n= 750$ & 0.0578 & 0.0460 & 0.0355  \\
    $n = 1000$ & 0.0498 &  0.0480 & 0.0430 \\
		\hline
    \multicolumn{1}{c|}{} & \multicolumn{3}{c}{$d = 100$} \\
		\hline \hline
		 & $SLRT$ & $EM (5, 3)$ & $EM (10, 6)$ \\
		\hline
		$n = 100$ & 0.0614 & 0.0520 & 0.0450   \\
		$n = 250$ & 0.0660& 0.0520 & 0.0440  \\
		$n = 500$  & 0.0594 & 0.0465 &0.0490 \\
    $n= 750$ & 0.0562 & 0.0470 & 0.0470 \\
    $n = 1000$ & 0.0494 &  0.0540 & 0.0560 \\
		\hline
	\end{tabular} \\
	\begin{flushleft}
	\small
  Notes: The columns ``$SLRT$'' report the rejection frequencies of the shrinkage likelihood ratio test, while ``$EM (J, K)$'' report those of the EM test with $J$ initial values and the step size $K$.
  The number of replications is 5000 for $SLRT$, and is 2000 for ``$EM (J, K)$.''
	\end{flushleft}
\end{table}

We move on to the evaluation of the power property of the two tests. 
The distribution of the data is the same as the alternative setting in subsection 4.2 with $(Z_{(2)}, \dots, Z_{(d)})' \sim N(0, I)$, except that we now consider two scenarios for the true values of $\gamma$: $(i) \ \gamma = (0.001, \dots, 0.001)'$ and $(ii) \ \gamma = (1, \dots, 1)'$. As detailed in Section 2.2.2 of \cite{shen2015inference}, the EM test does not cover the scenario $(i)$ because this test requires the absolute value of at least one element in $\gamma$ to be bounded away from a prespecified positive constant $c_1$, which is set to 0.05 in our computation. Hence, by including scenario $(i)$,  we can empirically assess whether the proposed method can serve as a viable alternative in a context for which the EM test is not designed.
The level is set to $0.05$. 
The implementation detail for the shrinkage likelihood ratio test and the EM test is the same as above.
For fair comparison, we use the size-adjusted critical values, obtained from the simulation under the null setting, for the two tests.
According to Table \ref{comparison size-adjusted power i} and \ref{comparison size-adjusted power ii}, the shrinkage likelihood ratio test demonstrates stable performance in that the test retains its power as $d$ increases in both scenarios. 
With regard to the EM test, Table \ref{comparison size-adjusted power i} documents its unsatisfactory performance in scenario $(i)$. Given the fact that the larger sample size brings hardly any power increase, the EM test seems unsuited for this setting with $\gamma$ close to zero, which aligns with the test's restriction on the value of $\gamma$. 
Although the performance is not as poor as in scenario $(i), $Table \ref{comparison size-adjusted power ii} reveals that the EM test suffers significant power loss from the high-dimensionality in scenario $(ii)$.
Specifically, when $d = 10$, the EM test shows a superior performance compared to that of the shrinkage likelihood ratio test.
However, such a superior performance of the former test rapidly deteriorates as $d$ increases to $50$ or $100$ though increases in the number of the initial values $J$ and the step size $K$ bring some power gain.
This might be because, as $d$ increases, it is more likely that none of the $J$ initial values is close enough to the true value $\gamma = (1, \dots, 1)'$, and, thus, the EM steps do not increase the test statistics in a satisfactory way.
We note that there is a possibility that sufficiently large values of $(K, J)$ get the EM test as powerful as the shrinkage likelihood ratio test for $d = 50$ or $d = 100$. For example, if we select the $J$ initial values such that they cover all the orthants of $\Gamma$, as detailed in Section 2.2.2 of \cite{shen2015inference}, one of these $J$ values will fall into the same orthant as the true $\gamma$.
This could serve as an effective starting value, thereby enhancing the test's power. 
However, this strategy requires $J$ to be $2^d$, which is computationally infeasible when $d = 50$ or $d = 100$.
Although there might be a more practical approach in selecting $J$,  we would still expect increasing the power to come at a heavy computational cost in high-dimensional settings.
\begin{table}[t]\caption{Comparison of Size-adjusted Power in Scenario $(i)$} \label{comparison size-adjusted power i} \medskip 
	\centering 
	\begin{tabular}{c|c|cc}
		\hline
		\multicolumn{1}{c|}{} & \multicolumn{3}{c}{$d = 10$} \\
		\hline \hline
		 & $SLRT$ & $EM (5, 3)$ & $EM (10, 6)$ \\
		\hline
		$n = 100$ & 0.1782 & 0.1060 & 0.1095   \\
		$n = 250$ & 0.3456& 0.0985 & 0.1185  \\
		$n = 500$  & 0.5398  & 0.1025 & 0.1165 \\
    $n= 750$ & 0.7178 & 0.0990 & 0.1255  \\
    $n = 1000$ & 0.8078 &  0.1090 & 0.1080 \\
		\hline
    \multicolumn{1}{c|}{} & \multicolumn{3}{c}{$d = 50$} \\
		\hline \hline
		 & $SLRT$ & $EM (5, 3)$ & $EM (10, 6)$ \\
		\hline
		$n = 100$ & 0.1756 &0.0735 & 0.0855   \\
		$n = 250$ & 0.3160& 0.0985 & 0.0810  \\
		$n = 500$  & 0.5562 &0.0775 & 0.0970 \\
    $n= 750$ & 0.6906 & 0.0870 & 0.1185  \\
    $n = 1000$ & 0.8136 &  0.0875 & 0.1025 \\
		\hline
    \multicolumn{1}{c|}{} & \multicolumn{3}{c}{$d = 100$} \\
		\hline \hline
		 & $SLRT$ & $EM (5, 3)$ & $EM (10, 6)$ \\
		\hline
		$n = 100$ & 0.1898 & 0.0625 & 0.0915   \\
		$n = 250$ &  0.3126& 0.0810 & 0.1005 \\
		$n = 500$  & 0.5492 & 0.0895  & 0.0750 \\
    $n= 750$ & 0.6896 & 0.0740 & 0.0995 \\
    $n = 1000$ &0.8202 &  0.0770 & 0.0740 \\
		\hline
	\end{tabular} \\
	\begin{flushleft}
	\small
  Notes: The columns ``$SLRT$'' report the rejection frequencies of the shrinkage likelihood ratio test, while ``$EM (J, K)$'' report those of the EM test with $J$ initial values and the step size $K$.
  The number of replications is 5000 for $SLRT$, and is 2000 for ``$EM (J, K)$.''
	\end{flushleft}
\end{table}
\begin{table}[t]\caption{Comparison of Size-adjusted Power in Scenario $(ii)$} \label{comparison size-adjusted power ii} \medskip 
	\centering 
	\begin{tabular}{c|c|cc}
		\hline
		\multicolumn{1}{c|}{} & \multicolumn{3}{c}{$d = 10$} \\
		\hline \hline
		 & $SLRT$ & $EM (5, 3)$ & $EM (10, 6)$ \\
		\hline
		$n = 100$ & 0.1768 & 0.2400 & 0.2620   \\
		$n = 250$ & 0.3692& 0.4430 & 0.5815  \\
		$n = 500$  & 0.6538  & 0.6805 & 0.8640 \\
    $n= 750$ & 0.8248 & 0.8220 & 0.9445  \\
    $n = 1000$ & 0.9240 &  0.8905 & 0.9770 \\
		\hline
    \multicolumn{1}{c|}{} & \multicolumn{3}{c}{$d = 50$} \\
		\hline \hline
		 & $SLRT$ & $EM (5, 3)$ & $EM (10, 6)$ \\
		\hline
		$n = 100$ & 0.1736 &0.0980 & 0.1125   \\
		$n = 250$ & 0.3462& 0.1595 & 0.1665  \\
		$n = 500$  & 0.6104 &0.2215 & 0.3030 \\
    $n= 750$ & 0.7386 & 0.3140 & 0.4640  \\
    $n = 1000$ & 0.8490 &  0.3965 & 0.5565 \\
		\hline
    \multicolumn{1}{c|}{} & \multicolumn{3}{c}{$d = 100$} \\
		\hline \hline
		 & $SLRT$ & $EM (5, 3)$ & $EM (10, 6)$ \\
		\hline
		$n = 100$ & 0.2008 & 0.0740 & 0.1055   \\
		$n = 250$ & 0.3570& 0.0950 & 0.1220 \\
		$n = 500$  & 0.5896 & 0.1420 &0.1515 \\
    $n= 750$ & 0.7302 & 0.1685 & 0.2090 \\
    $n = 1000$ & 0.8364 &  0.1880 & 0.2390 \\
		\hline
	\end{tabular} \\
	\begin{flushleft}
	\small
  Notes: The columns ``$SLRT$'' report the rejection frequencies of the shrinkage likelihood ratio test, while ``$EM (J, K)$'' report those of the EM test with $J$ initial values and the step size $K$.
  The number of replications is 5000 for $SLRT$, and is 2000 for ``$EM (J, K)$.''
	\end{flushleft}
\end{table}

To further examine the trade-off between power gain and computational time for the EM test, we conduct an additional simulation.
Specifically, we track changes in power and computational time as $J$ increases. We consider the same alternative setting as in the previous paragraph with $\gamma = (1, \dots, 1)'$, $n = 1000$ and $d \in \{ 50, 100 \}$. Due to the computational burden, we reduce the number of replications and the bootstrap sample sizes to 1000 and 200, respectively, for these new simulations. Our focus is on the change in $J$ because, as shown in Table \ref{change in K}, there is only marginal power gain, if any, with increase in $K$. 
Table \ref{power change in J} documents how the power of the EM test changes as $J$ increases from $10$ to $320$, or until its power exceeds that of the proposed method. 
As the table indicates, the rate at which power increases is slow compared to the increase in $J$. Specifically, for $d = 100$, the power of the EM test is approximately 20 \% below that of the proposed method, even with the maximum number of $J$ in the table. From a computational perspective, as shown in Table \ref{time change in J}, the computational time of the EM test is considerably longer than that of the proposed method and is essentially linear in $J$. These two factors, the slow increase in power and the heavy computational burden proportional to $J$, underscore the practical difficulty of the EM test attaining comparable power to that of the proposed method in high-dimensional settings.

\begin{table}[t]
	\caption{Change in Power as $K$ increases}
	\label{change in K} \medskip
	\centering 
	\begin{tabular}{c|cccc}
		\hline 
		& $K=6$ & $K=12$ & $K=24$ & $K=48$ \\ \hline \hline
		$d=50$ & 0.5565  & 0.5400 & 0.5400 & 0.5330  \\ \hline 
        $d=100$ & 0.2390 & 0.2760 & 0.2810 & 0.2870 \\ \hline
	\end{tabular} \\
	\small
    \begin{flushleft}
	Notes: Each entry shows the rejection frequency for the corresponding $K$ and $d$ with $J$ set to $10$. We use the size-adjusted critical values obtained from simulations for the corresponding null settings, which are available at the author's GitHub repository.
    \end{flushleft}. 
\end{table}

\begin{table}[t]
	\caption{Change in Power as $J$ increases}
	\label{power change in J} \medskip
	\centering 
	\begin{tabular}{c|c|cccccc}
		\hline 
		& $SLRT$ & $EM (10)$ & $EM(20)$ & $EM (40)$ & $EM (80)$ & $EM(160)$ & $EM(320)$ \\ \hline \hline
		$d=50$ & 0.8490  & 0.5565 & 0.6400 & 0.7760 & 0.8510  \\ \hline 
        $d=100$ & 0.8364 & 0.2390 & 0.2840 & 0.4090  & 0.4760 & 0.5150 & 0.6490\\ \hline
	\end{tabular} \\
	\small
    \begin{flushleft}
	Notes: The column ``$SLRT$'' reports the rejection frequencies of the shrinkage likelihood ratio test. The columns ``$EM (m)$'' report the rejection frequencies of the EM test with $J = m$ and $K = 6$. We use the size-adjusted critical values obtained from simulations for the corresponding null settings, which are available at the author's GitHub repository.
    \end{flushleft}. 
\end{table}

\begin{table}[t]
	\caption{Computational Time}
	\label{time change in J} \medskip
	\centering 
	\begin{tabular}{c|c|cccccc}
		\hline 
		& $SLRT$ & $EM (10)$ & $EM(20)$ & $EM (40)$ & $EM (80)$ & $EM(160)$ & $EM(320)$ \\ \hline \hline
        $d=50$ &2.9   & 43.3  &  71.7   & 118.8  & 208.4   &  & \\ \hline
		$d=100$ &2.5  & 54.1 & 81.2 & 128.1  & 217.4 & 383.6  & 704.0    \\ \hline 
	\end{tabular} \\
	\small
    \begin{flushleft}
	Notes: The column ``$SLRT$'' reports the time (in seconds) required to complete a single shrinkage likelihood ratio test. The columns ``$EM (m)$'' report the time (in seconds) needed for a single EM test with $J = m$ and $K = 6$, using a bootstrap sample size of $500$ to reflect practical situations. We do not measure the time for $EM (160)$ and $EM (320)$ with $d = 50$, based on the result from Table \ref{power change in J}. Our computing environment is Apple M2 Pro (12 cores) and 16 GB RAM.
    \end{flushleft}. 
\end{table}

\section{Real-World Data Analysis}
We apply the proposed method to data from AIDS Clinical Trials Group Protocol 175 (ACTG175), which is available in R package \textbf{speff2trial}.
This study randomizes 2139 HIV-infected patients into four different treatment arms: zidovudine (ZDV) only, ZDV plus didanosine (ddI), ZDV plus zalcitabine (zal), and ddI only.
As common in previous works \citep[e.g.,][]{lu2013variable,fan2017change}, we consider the CD4 count at $20 \pm 5$ weeks after the randomization as the outcome variable of interest.
Following \cite{lu2013variable}, we include the following 12 covariates plus the intercept term for $X$ in \eqref{model}: age, weight, Karnofsky score, CD4 count at baseline, CD8 count at baseline, hemophilia, homosexual activity, history of intravenous drug use, race, gender, antiretroviral history, and symptomatic status. 
To evaluate performance in a high-dimensional setting, we use 12 covariates in $X$ and their two-way interaction terms, in addition to the intercept term, for $Z$. 
As a result, the dimension of $Z$ is 79. Note that we standardize all covariates in $Z$ except the intercept term.
For the treatment variable $D$, we conduct the following four analyses as in \cite{lu2013variable}.

\begin{itemize} 
  \item $Analysis \ 1$: $D = 0$ for ZDV only and $D = 1$ for the other three treatment combined together.
  \item $Analysis \ 2$: Consider only those with ZDV plus ddI or ZDV plus zal.
  $D = 0$ for ZDV plus zal and $D = 1$ for ZDV plus ddI.
  \item $Analysis \ 3$: Consider only those with ZDV plus ddI or ddI only. 
  $D = 0$ for ddI only and $D = 1$ for ZDV plus ddI.
  \item $Analysis \ 4$: Consider only those with ZDV plus zal or ddI only.
  $D = 0$ for ddI only and $D = 1$ for ZDV plus zal.
\end{itemize}
\cite{fan2017change} also test the existence of a subgroup; however, they include only two covariates (age and homosexual activity) for $X$ and $Z$ and are limited to Analysis 2.

The results are summarized in Table \ref{real data}.
The proposed test rejects the null hypothesis of no subgroup at the 5\% level for Analysis 1-3, while it fails to reject the null hypothesis for Analysis 4.
In particular, the $p$-value for Analysis 2 is less than 0.001, suggesting strong evidence for the existence of a subgroup, which is in accordance with the finding of \cite{fan2017change}.
Furthermore, even though \cite{lu2013variable} do not perform the hypothesis test, they suggest the absence of treatment heterogeneity explained by covariate for Analysis 4 based on their estimation result.
This conclusion is consistent with our failure to reject the null hypothesis for Analysis 4.

\begin{table}[!htb]
	\caption{The $p$-values for the real-world data analysis}
	\label{real data} \medskip
	\centering 
	\begin{tabular}{c|cccc}
		\hline
		& Analysis 1 & Analysis 2 & Analysis 3 & Analysis 4 \\ \hline \hline
		$p$-value & 0.0098 & 0.0005 & 0.0067 & 0.5 \\ \hline 
	\end{tabular} \\
	\small
	Notes: Each entry shows the $p$-value of $SLRT$ for the corresponding analysis.
\end{table}

\section{Conclusion}
This study develops a new shrinkage testing method for the existence of a subgroup with a differential treatment effect in logistic-normal mixture models.
Compared to the existing works, the proposed test is computationally easy to implement due to the tractable asymptotic null distribution of the test statistics 
and accommodates high-dimensional covariates characterizing the classification of the subgroup.
Furthermore, we confirm the good finite sample performance of the proposed method through numerical simulation.

There are several interesting directions for future research.
First, the theoretical properties of the test statistics under the alternative hypothesis remain unknown.
Even though we expect the presence of $L_1$-penalization to complicate the situation, the theoretical analysis of the power merits investigation in order to fully characterize the performance of the proposed method.
Second, the proposed test hinges on the correctness of the parametric specification of the model, which can be too restrictive in real data.
Hence, the development of a specification test or extension of the proposed shrinkage method to more general nonlinear models as in \cite{andrews2012estimation} is an important research topic.
Lastly, in clinical trials, it is often the case that the outcome of interest is survival time and thus possibly right-censored.
Tailoring the proposed method for such survival data is of practical importance.

\clearpage
\appendix

\section{Proofs}
We employ the notations in the empirical process theory: for any random variable $R$ and measurable function $h$ in a class $\mathcal H$, we write $P h(R) = \mathbb E[h(R)]$ and $\mathbb P_n h(R) = \frac{1}{n} \sum_{i = 1}^n h(R_i)$,
and let $\| \cdot \|_{\mathcal H}$ denote the supremum of the absolute value over $\mathcal H$. For example, $\| \mathbb P_n h(R) - P h(R) \|_{\mathcal H} = \sup_{h \in \mathcal H} |\mathbb P_n h(R) - P h(R)|$.

In all the proofs in Appendix A and the proofs of Lemma \ref{consistency lemma}, \ref{maximal} and \ref{shrinkage pi} in Appendix B, we use the following Assumption \ref{parameter compact} in place of Assumption \ref{parameter}.
Assumptions \ref{parameter compact}$(a)$, $(c)$ and $(e)$ are identical to Assumptions \ref{parameter}$(a)$, $(c)$ and $(e)$, while Assumptions \ref{parameter compact}$(b)$ and $(d)$ replace $\Theta^{\sigma^2}$ and $\Gamma$ in Assumptions \ref{parameter}$(b)$ and $(d)$ with their compactified versions $\tilde \Theta^{\sigma^2}$ and $\Gamma_M$.
Lemma \ref{compact} in Appendix B shows that the replacement by Assumption \ref{parameter compact} is valid with probability approaching one under Assumptions \ref{parameter} and \ref{covariate}.
Let $\tilde \Theta := \Theta^{\alpha} \times \Theta^{\beta} \times \Theta^{\lambda} \times \tilde \Theta^{\sigma^2}$ denote the compactified parameter space.
\begin{assumptionp}{\ref*{parameter}$'$} \label{parameter compact}
  (a) $\Theta^{\alpha}$ and $\Theta^{\beta}$ are compact, convex sets,
  (b) $\tilde \Theta^{\sigma^2} = [l_{\sigma^2_0}, u_{\sigma^2_0}]$ for some $0 < l_{\sigma^2_0} < \sigma^2_0 < u_{\sigma^2_0} < \infty$,
  (c) $\Theta^{\lambda} = [0, u_{\lambda}]$ for some $0 < u_{\lambda} < \infty$,
  (d) $\Gamma_M = \{ \gamma \in \mathbb R^{d} : \| \gamma \|_1 \leq M n / p \}$ for some $M$, and
  (e) $(\alpha_0', \beta_0)'$ lies in an interior of $\Theta^{\alpha} \times \Theta^{\beta}$.
\end{assumptionp}

\begin{proof} [Proof of Proposition \ref{consistency}]
  We apply Lemma \ref{consistency lemma} with $c_n = n^{-1/4} (\log d \log n)^{-1/2}$.
  To this end, we divide the proof into three steps: $(i)$ characterization of $a_n$, $(ii)$ characterization of $b_{\ve, n}$, 
  and $(iii)$ verification of $a_n = o(b_{\ve, n})$.

  \underline{$(i)$ characterization of $a_n$}.
  Let $\tilde f(Y | W; \theta, \gamma) := (2 \pi \sigma^2 )^{1/2} f(Y | W; \theta, \gamma) =  \pi (Z'\gamma) e^{- \frac{(Y - X'\alpha - D(\beta + \lambda))^2}{2\sigma^2}} + (1 - \pi(Z'\gamma)) e^{- \frac{(Y - X'\alpha - D\beta)^2}{2\sigma^2} }$.
  Then a straightforward calculation yields 
  \begin{equation} \label{prop consistency a_n}
    a_n = \mathbb E\left[\left\|(\mathbb P_n - P) \log \tilde f(Y|W; \theta, \gamma)\right\|_{\tilde \Theta \times \Gamma_M}\right].
  \end{equation}
  Let $\xi_1, \dots, \xi_n$ be $i.i.d.$ Rademacher random variables independent of $\{(Y_i, W_i) \}_{i = 1}^n$ (see Lemma 2.2.7 of \cite{van1996weak} for the definition).
  Lemma 2.3.1 of \cite{van1996weak} (the symmetrization inequality) gives that 
  \begin{equation} \label{prop consistency symmetrization}
    \mathbb E\left[\left\| (\mathbb P_n - P) \log \tilde f(Y|W; \theta, \gamma)\right\|_{\tilde \Theta \times \Gamma_M}\right] \lesssim \mathbb E\left[\left\|  \mathbb P_n \xi \log \tilde f(Y|W; \theta, \gamma) \right\|_{\tilde \Theta \times \Gamma_M}\right].
  \end{equation}
  Letting $\theta_{\ast} := (\alpha_{\ast}', \beta_{\ast}, \lambda_{\ast}, l_{\sigma^2_0})'$ with $\alpha_{\ast} = 0$, $\beta_{\ast} = 0$, $\lambda_{\ast} = 0$ and $l_{\sigma^2_0}$ defined in Assumption \ref{parameter compact}$(b)$, by the triangle inequality,
  \begin{align} \label{prop consistency triangle}
    \mathbb E\left[\left\|  \mathbb P_n \xi \log \tilde f(Y|W; \theta, \gamma) \right\|_{\tilde \Theta \times \Gamma_M}\right] \leq \ &\mathbb E\left[\left\| \mathbb P_n \xi \left( \log \tilde f(Y | W; \theta, \gamma) - \log \tilde f(Y | W; \theta_{\ast}, 0) \right) \right\|_{\tilde \Theta \times \Gamma_M}\right] \notag \\
    &+ \mathbb E\left[\left| \mathbb P_n \xi \log \tilde f(Y | W; \theta_{\ast}, 0)) \right|\right].
  \end{align}

  We bound each of the two terms on the right side. For the first term, we start by considering the function $g(w_1, w_2, w_3) := \log (\pi (w_1) e^{-w_2} + (1 - \pi (w_1)) e^{-w_3})$.
  By a straightforward calculation using differentiation and the mean value theorem,
  $|g(w_1, w_2, w_3) - g(u_1, u_2, u_3)| \leq |w_1 - u_1| + |w_2 - u_2| + |w_3 - u_3|$ for any $(w_1, w_2, w_3)$ and $(u_1, u_2, u_3) \in \mathbb R^3$.
  It now follows from Lemma \ref{multivariate contraction} that 
  \begin{align} \label{prop consistency multivariate}
    &\mathbb E\left[\left\| \mathbb P_n \xi \left( \log \tilde f(Y | W; \theta, \gamma) - \log \tilde f(Y | W; \theta_{\ast}, 0) \right) \right\|_{\tilde \Theta \times \Gamma_M}\right] \notag \\
    \lesssim \ &\mathbb E \left[ \left\| \mathbb P_n \omega Z'\gamma \right\|_{\Gamma_M} \right] + \mathbb E\left[ \left\| \mathbb P_n \omega \frac{(Y - X'\alpha - D(\beta + \lambda))^2}{2 \sigma^2} \right\|_{\tilde \Theta} \right]
    + \mathbb E \left[ \left\| \mathbb P_n \omega \frac{(Y - X'\alpha - D\beta)^2}{2 \sigma^2} \right\|_{\tilde \Theta} \right],
  \end{align}
  where $\omega$ is a standard normal random variable as defined in the statement of Lemma \ref{multivariate contraction}.
  The right side is further bounded by $\mathcal C (\sqrt{n \log d} / p + n^{-1/2})$ from Lemma \ref{maximal}, by which we obtain
  \begin{equation} \label{prop consistency a first bound}
    \mathbb E\left[\left\| \mathbb P_n \xi \left( \log \tilde f(Y | W; \theta, \gamma) - \log \tilde f(Y | W; \theta_{\ast}, 0) \right) \right\|_{\tilde \Theta \times \Gamma_M}\right] \lesssim \frac{\sqrt{n \log d}}{p} + \frac{1}{\sqrt{n}}.
  \end{equation}

  For the second term on the right side of \eqref{prop consistency triangle}, Lemma 8 of \cite{chernozhukov2015comparison} combined with Assumption \ref{covariate}$(a)$ gives that $\mathbb E\left[\left| \mathbb P_n \xi \log \tilde f(Y | W; \theta_{\ast}, 0)) \right|\right] \lesssim n^{-1/2}.$
  In light of this inequality and \eqref{prop consistency a first bound}, $a_n \lesssim (n \log d)^{1/2}/p + n^{-1/2}$ follows from \eqref{prop consistency a_n}, \eqref{prop consistency symmetrization} and \eqref{prop consistency triangle}.

  \underline{$(ii)$ characterization of $b_{\ve, n}$.}
  Let $r_n$ be an arbitrary sequence of positive real numbers converging to zero.
  Define $\Psi^{r_n}_{\ve} := \{ (\theta', \gamma)' \in \tilde \Theta \times \Gamma_M : \| \theta - \theta_0 \| + c_n^{-1} \| \gamma \|_1 \geq \ve, c_n^{-1} \| \gamma \|_1 \leq r_n \}$
  and $\Psi^{r_n}_{\ve, c} := \{ (\theta', \gamma)' \in \tilde \Theta \times \Gamma_M : \| \theta - \theta_0 \| + c_n^{-1} \| \gamma \|_1 \geq \ve, c_n^{-1} \| \gamma \|_1 > r_n \}$.
  Then, because $\Xi_{\ve, n} = \Psi^{r_n}_{\ve} \cup \Psi^{r_n}_{\ve, c}$ ($\Xi_{\ve, n}$ is defined in the statement of Lemma \ref{consistency lemma}),
  $b_{\ve, n}$ is no smaller than
  \begin{equation} \label{prop consistency b bound}
    \mathbb E[\log f (Y | W; \theta_0, 0)] - (\mathcal A_n \lor \mathcal B_n) 
    \geq (\mathbb E[\log f (Y | W; \theta_0, 0)] - \mathcal A_n) \land (\mathbb E[\log f (Y | W; \theta_0, 0)] - \mathcal B_n),
  \end{equation}
  where
  \begin{equation} 
    \mathcal A_n := \sup_{(\theta', \gamma')' \in \Psi^{r_n}_{\ve}} \left( \mathbb E[\log f(Y | W; \theta, \gamma) ] - p/n \| \gamma \|_1 \right), \
    \mathcal B_n := \sup_{(\theta', \gamma')' \in \Psi^{r_n}_{\ve, c}} \left( \mathbb E[\log f(Y | W; \theta, \gamma) ] - p/n \| \gamma \|_1 \right) \notag.
  \end{equation}

  We bound each of $\mathbb E[\log f (Y | W; \theta_0, 0)] - \mathcal A_n$ and $\mathbb E[\log f (Y | W; \theta_0, 0)] - \mathcal B_n$ from below.
  For the first quantity, a straightforward calculation gives that $\mathbb E[\log f (Y | W; \theta_0, 0)] - \mathcal A_n$ is no smaller than 
  \begin{equation} \label{prop consistency b a_n}
    \mathbb E [\log f(Y | W; \theta_0, 0)] - \sup_{(\theta', \gamma')' \in \Psi^{r_n}_{\ve}} \mathbb E[\log f(Y | W; \theta, 0)] - \left\| \mathbb E[\log f(Y | W; \theta, \gamma)] - \mathbb E[\log f(Y | W; \theta, 0)] \right\|_{\Psi^{r_n}_{\ve}}.
  \end{equation}
  Note that for sufficiently large $n$ with $r_n \leq \ve / 2$, $(\theta', \gamma')' \in \Psi^{r_n}_{\ve}$ implies that $\| \theta - \theta_0 \| \geq \ve / 2$.
  Hence, for such $n$, it holds that
  \begin{align} \label{prop consistency b a_n first}
    &\mathbb E [\log f(Y | W; \theta_0, 0)] - \sup_{(\theta', \gamma')' \in \Psi^{r_n}_{\ve}} \mathbb E[\log f(Y | W; \theta, 0)] \notag \\
    \geq \ &\mathbb E [\log f(Y | W; \theta_0, 0)] - \sup_{\theta \in \{ \theta \in \Theta: \| \theta - \theta_0 \| \geq \ve / 2 \} } \mathbb E[\log f(Y|W; \theta, 0)],
  \end{align} 
  where the right side is positive from the information inequality combined with the model identifiability and Assumption \ref{parameter compact}$(a)$-$(c)$.
  Meanwhile, similarly to the argument leading to \eqref{prop consistency multivariate}, $|\log f(Y | W; \theta, \gamma) - \log f(Y | W; \theta, 0)| \leq |Z'\gamma|$, which is bounded by $\| \gamma \|_1 \max_{1 \leq j \leq d} |Z_{(j)}|$.
  Consequently,
  \begin{equation} \label{prop consistency b a_n second}
    \left\| \mathbb E[\log f(Y | W; \theta, \gamma)] - \mathbb E[\log f(Y | W; \theta, 0)] \right\|_{\Psi^{r_n}_{\ve}} \leq c_n r_n \mathbb E \left[\max_{1 \leq j \leq d} |Z_{(j)}|\right] \lesssim c_n r_n \sqrt{\log d} = o(1),
  \end{equation}
  where the second inequality follows from Lemma 2.2.1 and 2.2.2 of \cite{van1996weak} and Assumption \ref{covariate}$(c)$ and the last equality follows from the choice of $c_n$.
  Combining \eqref{prop consistency b a_n}, \eqref{prop consistency b a_n first} and \eqref{prop consistency b a_n second}, we obtain 
  \begin{equation} \label{prop consistency b a_n bound}
    \mathbb E[\log f (Y | W; \theta_0, 0)] - \mathcal A_n \geq c_{\ve},
  \end{equation}
  for some positive constant $c_{\ve}$ for sufficiently large $n$.

  For $\mathbb E[\log f (Y | W; \theta_0, 0)] - \mathcal B_n$, note that 
  \begin{equation} 
    \mathbb E[\log f (Y | W; \theta_0, 0)] - \mathcal B_n 
    \geq \mathbb E [\log f(Y|W; \theta_0, 0)] - \sup_{(\theta', \gamma')' \in \Psi^{r_n}_{\ve, c}} \mathbb E[\log f(Y | W; \theta, \gamma)] + \inf_{(\theta', \gamma')' \in  \Psi^{r_n}_{\ve, c}} p/n \| \gamma \|_1, \notag
  \end{equation}
  where $\mathbb E [\log f(Y|W; \theta_0, 0)] - \sup_{(\theta', \gamma')' \in \Psi^{r_n}_{\ve, c}} \mathbb E[\log f(Y | W; \theta, \gamma)] \geq 0$ from the information inequality and the model identifiability, and 
  $\inf_{(\theta', \gamma')' \in  \Psi^{r_n}_{\ve, c}} p/n \| \gamma \|_1 \geq p r_n c_n / n = (p_n r_n) / (n^{5/4} \sqrt{\log d \log n})$ from the construction of $\Psi^{r_n}_{\ve, c}$ and the choice of $c_n$.
  As a result,
  \begin{equation} \label{prop consistency b b_n bound}
    \mathbb E[\log f (Y | W; \theta_0, 0)] - \mathcal B_n \geq (p_n r_n) / (n^{5/4} \sqrt{\log d \log n})
  \end{equation}
  holds.

  Combining \eqref{prop consistency b bound}, \eqref{prop consistency b a_n bound} and \eqref{prop consistency b b_n bound}, 
  we have $b_{\ve, n} \gtrsim  p_n r_n / (n^{5/4} \sqrt{\log d \log n})$.

  \underline{$(iii)$ verification of $a_n = o(b_{\ve, n})$}.
  Combine the results from $(i)$ and $(ii)$ to obtain the inequality 
  \begin{equation} 
    \frac{a_n}{b_{\ve, n}} \lesssim \frac{1}{r_n} \left( \frac{n^{7/4} \sqrt{\log n} \log d}{p^2} + \frac{n^{3/4} \sqrt{\log n \log d}}{p} \right) \lesssim \frac{1}{r_n} \left( \frac{n^{7/4} \sqrt{\log n} \log d}{p^2} + \sqrt{\frac{n^{7/4} \sqrt{\log n} \log d}{p^2}} \right). \notag
  \end{equation}
  Taking the convergence rate of $r_n$ to zero sufficiently slow, the right side converges to zero by Assumption \ref{tuning parameter}$(a)$.
  Lemma \ref{consistency lemma} now completes the proof.
\end{proof}

\begin{proof} [Proof of Proposition \ref{quad approx}]
  First, we prove part $(a)$.
  Let $\psi = (\zeta_1, \dots, \zeta_r)'$ with $r = q + 3$. 
  Collect $\eta_\sigma := (\eta', \sigma^2)'$, $t_n(\eta_\sigma) := n^{1/2} (\eta_\sigma - {\eta_{\sigma}}_0)$ and $s_{i}^{\eta_\sigma} := (U_i' H^1_i / \sigma_0, H^2_i / (2 \sigma^2_0)')'$. 
  Accordingly, define $S_n^{\eta_\sigma} := n^{-1/2} \sum_{i = 1}^n s^{\eta_\sigma}_i$ and $\mathcal I_n^{\eta_\sigma} := n^{-1} \sum_{i = 1}^n s_i^{\eta_\sigma} {s_i^{\eta_\sigma}}'$.
  We abbreviate $f (Y | W; \psi_0, \hat \gamma^*)$ to $f_0$.
  Let $\{ \eta_\sigma \}$ be a set consisting of the elements of $\eta_\sigma$.
  Furthermore, let $V$ and $\rho (\ve)$ denote a random variable with the finite second moments that is independent of $\ve$ and a polynomial of $\ve$ whose value and form may vary from one expression to another, respectively.
  Expanding $l_n^{\ast} (\psi, \hat \gamma^*)$ around $\psi_0$ five times using Taylor's theorem yields
  $l^{\ast}_n (\psi, \hat \gamma^*) - l^{\ast}_n (\psi_0, \hat \gamma^*) = \sum_{k = 1}^5 \mathcal D_k$ where
  \begin{align}
    &\mathcal D_1 := \nabla_{\psi} l_n (\psi_0, \hat \gamma^*)' (\psi - \psi_0), \ \mathcal D_2 := \frac{1}{2} (\psi - \psi_0)' \nabla_{\psi \psi'} l_n (\psi_0, \hat \gamma^*) (\psi - \psi_0) \notag \\
    &\mathcal D_3 := \frac{1}{3!} \sum_{i = 1}^r \sum_{j = 1}^r \sum_{k = 1}^r \nabla_{\zeta_i \zeta_j \zeta_k} l_n (\psi_0, \hat \gamma^*) (\zeta_i - \zeta_{i, 0})(\zeta_j - \zeta_{j, 0})(\zeta_k - \zeta_{k, 0}), \notag \\
    &\mathcal D_4 := \frac{1}{4!} \sum_{i = 1}^r \sum_{j = 1}^r \sum_{k = 1}^r \sum_{l = 1}^r \nabla_{\zeta_i \zeta_j \zeta_k \zeta_l} l_n (\psi_0, \hat \gamma^*) (\zeta_i - \zeta_{i, 0})( \zeta_j - \zeta_{j, 0})(\zeta_k - \zeta_{k, 0})(\zeta_l - \zeta_{l, 0}), \notag \\
    &\mathcal D_5 := \frac{1}{5!} \sum_{i = 1}^r \sum_{j = 1}^r \sum_{k = 1}^r \sum_{l = 1}^r \sum_{m = 1}^r \nabla_{\zeta_i \zeta_j \zeta_k \zeta_l \zeta_m} l_n (\bar \psi, \hat \gamma^*) (\zeta_i - \zeta_{i, 0})(\zeta_j - \zeta_{j, 0})(\zeta_k - \zeta_{k, 0})(\zeta_l - \zeta_{l, 0}) (\zeta_m - \zeta_{m, 0}), \notag
  \end{align}
  where $\bar \psi$ lies on the path connecting $\psi_0$ and $\psi$. We now investigate each of $\mathcal D_1$-$\mathcal D_5$.

  \underline{(i) $\mathcal D_1$}.
  By a straightforward calculation using Lemma \ref{hermite}, 
  \begin{equation} 
    \mathcal D_1 = (S_n^{\eta_\sigma})' t_n(\eta_{\sigma}) + \left( \frac{1}{n^{1/4}} \sum_{i = 1}^n (2\pi (Z_i'\hat \gamma^*) - 1) \frac{D_i H_i^1}{2 \sigma_0} \right) n^{1/4} \lambda = (S_n^{\eta_\sigma})' t_n(\eta_\sigma) + o_p (1) n^{1/4} \lambda, \notag
  \end{equation}
  where the second equality follows from Lemma \ref{shrinkage pi}$(a)$.

  \underline{(ii) $\mathcal D_2$}.
  Let $\mu_1, \mu_2$ be any elements of $\psi$. Then $\nabla_{\mu_1 \mu_2} \log f_0 = \nabla_{\mu_1 \mu_2} f_0 / f_0 - (\nabla_{\mu_1} f_0 / f_0) (\nabla_{\mu_2} f_0 / f_0)$.
  When $\mu_1, \mu_2 \in \{\eta_\sigma \}$, a straightforward calculation with Lemma \ref{hermite} gives that $\mathbb E[\nabla_{\mu_1 \mu_2} f_0 / f_0] = 0$.
  In particular, this implies 
  \begin{equation} \label{prop quad approx d2 eta_sigma}
    (\eta_{\sigma} - {\eta_{\sigma}}_0)' \nabla_{\eta_\sigma \eta_\sigma'} l_n (\psi_0, \hat \gamma^*) (\eta_{\sigma} - {\eta_{\sigma}}_0) = t_n (\eta_\sigma)' o_p (1) t_n (\eta_\sigma) - t_n (\eta_\sigma)' \mathcal I^{\eta_\sigma}_n t_n (\eta_\sigma),
  \end{equation}
  where $o_p (1)$ follows from the law of large numbers. 
  At the same time, a straightforward calculation with Lemma \ref{hermite} gives that $(\eta_{\sigma} - {\eta_{\sigma}}_0)' \nabla_{\eta_{\sigma} \lambda} l_n (\psi_0, \hat \gamma^*) \lambda = t_n (\eta_\sigma)' S_n^{\eta_\sigma \lambda} n^{1/4} \lambda - t_n (\eta_{\sigma})' \mathcal I_n^{\eta_\sigma \lambda} n^{1/4} \lambda,$
  where we define $S_n^{\eta_\sigma \lambda} := n^{-3/4} \sum_{i = 1}^n \left((2 \pi (Z_i'\hat \gamma^*) - 1) U_i'D_i H_i^2/ (2 \sigma^2_0), (2 \pi (Z_i'\hat \gamma^*) - 1) D_i H^3_i / (4 \sigma^3_0)\right)'$
  and $\mathcal I_n^{\eta_\sigma \lambda} := n^{-3/4} \sum_{i = 1}^n s_i^{\eta_\sigma} (2 \pi (Z' \hat \gamma^*) - 1) D_i H_i^1 / (2 \sigma_0)$.
  By Lemma \ref{shrinkage pi}$(b)$, $S_n^{\eta_{\sigma} \lambda} = o_p(1)$ and $\mathcal I_n^{\eta_{\sigma} \lambda} = o_p(1)$.
  Hence, we have 
  \begin{equation} \label{prop quad approx d2 eta_sigma lambda}
    (\eta_{\sigma} - {\eta_{\sigma}}_0)' \nabla_{\eta_{\sigma} \lambda} l_n (\psi_0, \hat \gamma^*) \lambda^2 = t_n (\eta_{\sigma})' o_p (1) n^{1/4}\lambda.
  \end{equation}
  Lastly, by a straightforward calculation with Lemma \ref{hermite}, $\nabla_{\lambda^2} l_n (\psi_0, \hat \gamma^*) \lambda^2$ equals
  \[
    \left( n^{-1/2} \sum_{i = 1}^n D_i H^2_i / (4 \sigma^2_0) \right) n^{1/2} \lambda^2 - \left( n^{-1/2} \sum_{i = 1}^n (2\pi (Z_i'\hat \gamma^*) - 1)^2 (D_i H_i^1)^2 / (2 \sigma_0)^2 \right) n^{1/2} \lambda^2.
    \]
  However, observe that $n^{-1/2} \sum_{i = 1}^n (2\pi (Z_i'\hat \gamma^*) - 1)^2 (D_i H_i^1)^2 / (2 \sigma_0)^2 = o_p(1)$ by Lemma \ref{shrinkage pi}$(c)$.
  Hence, we have 
  \begin{equation} \label{prop quad approx d2 lambda}
    \nabla_{\lambda^2} l_n (\psi_0, \hat \gamma^*) \lambda^2 = \left( n^{-1/2} \sum_{i = 1}^n D_i H^2_i / (4 \sigma^2_0) \right) n^{1/2} \lambda^2 + o_p (1) n^{1/2} \lambda^2.
  \end{equation}
  Combining \eqref{prop quad approx d2 eta_sigma}, \eqref{prop quad approx d2 eta_sigma lambda} and \eqref{prop quad approx d2 lambda}, we obtain
  \begin{align} 
    \mathcal D_2 = \ &- \frac{1}{2} t_n (\eta_{\sigma})' \mathcal I_n^{\eta_{\sigma}} t_n (\eta_{\sigma}) + \left(n^{-1/2} \sum_{i = 1}^n D_i H^2_i / (8 \sigma_0^2) \right) n^{1/2} \lambda^2 \notag \\
    \ &+ t_n (\eta_{\sigma})'o_p(1) t_n (\eta_{\sigma}) + t_n (\eta_{\sigma})' o_p(1) n^{1/4} \lambda + o_p(1) n^{1/2} \lambda^2. \notag
  \end{align}

  \underline{(iii) $\mathcal D_3$}.
  Note that 
  \begin{align} \label{prop quad approx d3 decomp}
    \mathcal D_3 = \ &\frac{1}{2} \sum_{\zeta_i \in \{ \eta_\sigma \} } \sum_{\zeta_j \in \{ \eta_\sigma \} } \sum_{k = 1}^{r} \nabla_{\zeta_i \zeta_j \zeta_k} l_n (\psi_0; \hat \gamma^*) (\zeta_i - \zeta_{i, 0}) (\zeta_j - \zeta_{j, 0}) (\zeta_k - \zeta_{k, 0}) \notag \\
    \ &+ \frac{1}{2} \sum_{\zeta \in \{ \eta_{\sigma} \}} \nabla_{\zeta \lambda^2} l_n (\psi_0; \hat \gamma^*) (\zeta - \zeta_0) \lambda^2 + \frac{1}{3!} \nabla_{\lambda^3} l_n (\psi_0; \hat \gamma^*) \lambda^3.
  \end{align}
  We take a close look at each term on the right side.
  First, by a straightforward calculation with Lemma \ref{hermite} and Assumption \ref{covariate}$(a)$, for any elements $\mu_1, \mu_2, \mu_3$ of $\psi$, $| \nabla_{\mu_1 \mu_2 \mu_3} \log f_0|$ is bounded by a integrable function $g(W, \ve)$ that does not depend on $\hat \gamma^*$.
  Hence, by the law of large numbers,
  \begin{align} \label{prop quad approx d3 first}
    &\sum_{\zeta_i \in \{ \eta_\sigma \} } \sum_{\zeta_j \in \{ \eta_\sigma \} } \sum_{k = 1}^{r} \nabla_{\zeta_i \zeta_j \zeta_k} l_n (\psi_0; \hat \gamma^*) (\zeta_i - \zeta_{i, 0}) (\zeta_j - \zeta_{j, 0}) (\zeta_k - \zeta_{k, 0}) \notag \\
    = \ &\sum_{\zeta_i \in \{ \eta_\sigma \} } \sum_{\zeta_j \in \{ \eta_\sigma \} } O_p(1) n^{1/2} (\zeta_i - \zeta_{i, 0}) n^{1/2} (\zeta_j - \zeta_{j, 0})O(\| \psi - \psi_0 \|).
  \end{align}
  Subsequently, for $\zeta \in \{ \eta_{\sigma} \}$, a straightforward derivative calculation yields
  \[
    \nabla_{\zeta \lambda^2} \log f_0 = \frac{\nabla_{\zeta \lambda^2} f_0}{f_0} - 2  \frac{\nabla_{\zeta \lambda} f_0}{f_0} \frac{\nabla_{\lambda} f_0}{f_0} - \frac{\nabla_{\lambda^2} f_0}{f_0} \frac{\nabla_{\zeta} f_0}{f_0} + 2 \frac{\nabla_{\zeta} f_0}{f_0} \left( \frac{\nabla_{\lambda} f_0}{f_0} \right)^2.
  \]
  By Lemma \ref{hermite}, it is easy to verify that $\mathbb E [\nabla_{\zeta \lambda^2} f_0 / f_0] = 0$.
  Furthermore, Lemma \ref{hermite} also suggests that $(\nabla_{\zeta \lambda} f_0 / f_0) (\nabla_{\lambda} f_0 / f_0)$ and $(\nabla_{\zeta} f_0 / f_0) (\nabla_{\lambda} f_0 / f_0)^2$ can be written as the form $V (\pi (Z'\hat \gamma^*) - 1/2)^2 \rho (\ve)$.
  Hence, by the law of large numbers, Lemma \ref{shrinkage pi}$(b)$ and applying Lemma \ref{hermite} to $(\nabla_{\lambda^2} f_0 / f_0) (\nabla_{\zeta} f_0/f_0)$ yield
  \begin{align} \label{prop quad approx d3 second}
    \frac{1}{2} \sum_{\zeta \in \{ \eta_{\sigma} \}} \nabla_{\zeta \lambda^2} l_n (\psi_0; \hat \gamma^*) (\zeta - \zeta_0) \hat \lambda^2 = &\sum_{\zeta \in \{ \eta_{\sigma}  \}} o_p (1) n^{1/2} (\zeta - \zeta_0) n^{1/2} \lambda^2 \notag \\
    \ &- t_n (\eta_{\sigma})' \left( n^{-1} \sum_{i = 1}^n s_i^{\eta_\sigma} D_i H_i^2/(8 \sigma^2_0)  \right) n^{1/2} \lambda^2.
  \end{align}
  Lastly, by a straightforward derivative calculation with Lemma \ref{hermite}, $\nabla_{\lambda^3} \log f_0$ can be written as the form $D^3 (\pi (Z'\hat \gamma^*) - 1/2)^k \rho (\ve)$ with $k \in \mathbb N$.
  It follows from Lemma \ref{shrinkage pi}$(b)$ that $\nabla_{\lambda^3} l_n (\psi_0; \hat \gamma^*) \lambda^3 = o_p(1) n^{1/4} \lambda n^{1/2} \lambda^2$.
  Combining this with \eqref{prop quad approx d3 decomp}, \eqref{prop quad approx d3 first} and \eqref{prop quad approx d3 second} yields that 
  \begin{align} 
    \mathcal D_3 = &- t_n (\eta_{\sigma})' \left( n^{-1} \sum_{i = 1}^n s_i^{\eta_\sigma} D_i H_i^2/(8 \sigma^2_0)  \right) n^{1/2} \lambda^2 \notag \\
    \ &+ \sum_{\zeta_i \in \{ \eta_\sigma \} } \sum_{\zeta_j \in \{ \eta_\sigma \} } O_p(1) n^{1/2} (\zeta_i - \zeta_{i, 0}) n^{1/2} (\zeta_j - \zeta_{j, 0}) O(\| \psi - \psi_0 \|) \notag \\
    \ &+ \sum_{\zeta \in \{ \eta_{\sigma}  \}} o_p (1) n^{1/2} (\zeta - \zeta_0) n^{1/2} \lambda^2 + o_p(1) n^{1/4} \lambda n^{1/2} \lambda^2. \notag
  \end{align}

  \underline{(iv) $\mathcal D_4$}.
  Observe that 
  \begin{align} \label{prop quad approx d4 decomp}
    \mathcal D_4 = \ &\frac{1}{3} \sum_{\zeta \in \{ \eta_\sigma \}} \sum_{i = 1}^r \sum_{j = 1}^r \sum_{k = 1}^r \nabla_{\zeta \zeta_i \zeta_j \zeta_k} l_n (\psi_0; \hat \gamma^*) (\zeta - \zeta_0) (\zeta_{i} - \zeta_{i, 0}) (\zeta_j - \zeta_{j, 0}) (\zeta_k - \zeta_{k, 0}) \notag \\
    \ &+ \frac{1}{4!} \nabla_{\lambda^4} l_n (\psi_0; \hat \gamma^*) \lambda^4.
  \end{align}
  For the summation on the right side, a straightforward calculation with Lemma \ref{hermite} and Assumption \ref{covariate}$(a)$ gives that, for any elements $\mu_1, \mu_2, \mu_3, \mu_4$ of $\psi$, $|\nabla_{\mu_1 \mu_2 \mu_3 \mu_4} \log f_0|$ is bounded by a integrable function $g(W, \ve)$ that does not depend on $\hat \gamma^*$.
  Hence, by the law of large numbers,
  \begin{align} \label{prop quad approx d4 first}
    &\sum_{\zeta \in \{ \eta_\sigma \}} \sum_{i = 1}^r \sum_{j = 1}^r \sum_{k = 1}^r \nabla_{\zeta \zeta_i \zeta_j \zeta_k} l_n (\psi_0; \hat \gamma^*) (\zeta - \zeta_0) (\zeta_{i} - \zeta_{i, 0}) (\zeta_j - \zeta_{j, 0}) (\zeta_k - \zeta_{k, 0}) \notag \\
    = \ &\sum_{\zeta \in \{ \eta_\sigma\}} \sum_{\mu \in \{ \eta_\sigma, \lambda^2\}} O_p(1) n^{1/2} (\zeta - \zeta_0) n^{1/2} (\mu - \mu_0) O(\| \psi - \psi_0 \|),
  \end{align}
  where $\{ \eta_{\sigma}, \lambda^2 \}$ is a set consisting of the elements of $\eta_{\sigma}$ and $\lambda^2$.
  Furthermore, a straightforward derivative calculation yields
  \begin{equation} 
    \nabla_{\lambda^4} \log f_0 = \frac{\nabla_{\lambda^4} f_0}{f_0} - 4 \frac{\nabla_{\lambda^3} f_0}{f_0} \frac{\nabla_{\lambda} f_0}{f_0} - 3 \left( \frac{\nabla_{\lambda^2} f_0}{f_0} \right)^2 + 12 \frac{\nabla_{\lambda^2} f_0}{f_0} \left( \frac{\nabla_{\lambda} f_0}{f_0} \right)^2 - 6 \left( \frac{\nabla_{\lambda} f_0}{f_0} \right)^4. \notag 
  \end{equation}
  By Lemma \ref{hermite}, $\mathbb E[\nabla_{\lambda^4} f_0 / f_0] = 0$.
  Furthermore, a straightforward calculation with Lemma \ref{hermite} gives that $(\nabla_{\lambda^3} f_0/f_0)(\nabla_{\lambda} f_0/f_0)$, $(\nabla_{\lambda^2} f_0/f_0)(\nabla_{\lambda} f_0/f_0)^2$ and $(\nabla_{\lambda} f_0/f_0)^4$ all have the form $V (\pi (Z'\hat \gamma^*) - 1/2)^{k} \rho(\ve)$ where $k \in \mathbb N$ and $V$ has the finite second moment.
  Combining these facts with applying Lemma \ref{hermite} to $(\nabla_{\lambda^2} f_0 / f_0)^2 $ in conjunction with the law of large numbers and Lemma \ref{shrinkage pi}$(b)$,
  \begin{equation} \label{prop quad approx d4 second}
    \frac{1}{4!} \nabla_{\lambda^4} l_n (\psi_0, \hat \gamma^*) \lambda^4 = - \left(\frac{1}{2n} \sum_{i = 1}^n (D_i H_i^2 / (8 \sigma_0^2))^2\right) (n^{1/2} \lambda^2)^2 + o_p (1) (n^{1/2} \lambda^2)^2.
  \end{equation}
  In view of \eqref{prop quad approx d4 decomp}, \eqref{prop quad approx d4 first} and \eqref{prop quad approx d4 second}, we obtain 
  \begin{align} 
    \mathcal D_4 = \ &- \left(\frac{1}{2n} \sum_{i = 1}^n \left( \frac{D_i H_i^2}{8 \sigma_0^2} \right)^2\right) (n^{1/2} \lambda^2)^2 + \sum_{\zeta \in \{ \eta_\sigma\}} \sum_{\mu \in \{ \eta_\sigma, \lambda^2\}} O_p(1) n^{1/2} (\zeta - \zeta_0) n^{1/2} (\mu - \mu_0) O(\| \psi - \psi_0 \|) \notag \\
    \ &+ o_p (1) (n^{1/2} \lambda^2)^2. \notag
  \end{align}

  \underline{(v) $\mathcal D_5$}.
  For any elements $\mu_1, \dots, \mu_5$ of $\psi$, a straightforward calculation with Lemma \ref{hermite} in conjunction with Assumptions \ref{covariate}$(a)$ and \ref{parameter compact}$(a)$-$(c)$ implies that $|\nabla_{\mu_1 \dots \mu_5} \log f(Y|W; \bar \psi, \hat \gamma^*)|$ is bounded by an integrable function independent of the values of $\bar \psi$ and $\hat \gamma^*$.
  Hence, the law of large numbers implies that 
  \begin{equation} \label{prop quad approx d5}
    \mathcal D_5 = \sum_{\zeta \in \{ \eta_\sigma, \lambda^2 \}, \mu \in \{ \eta_\sigma, \lambda^2 \}} O_{p, \bar \psi}(1) n^{1/2} (\zeta - \zeta_0) n^{1/2} (\mu - \mu_0) O(\| \psi - \psi_0 \|), \notag
  \end{equation}
  where $|O_{p, \bar \psi}(1)| \leq O_p (1)$ for some $O_p (1)$ independent of $\bar \psi$ and $\hat \gamma^*$ as discussed above.

  Collecting the terms from $\mathcal D_1$-$\mathcal D_5$, we obtain 
  \begin{equation}
    l^{\ast}_n (\psi, \hat \gamma^*) - l_n^{\ast} (\psi_0, \hat \gamma^*) = S_n't_n(\psi) - \frac{1}{2} t_n (\psi)' \mathcal I_n t_n (\psi) + R_n (\psi, \hat \gamma^*), \notag
  \end{equation}
  where, by a straightforward calculation, $\sup_{\psi \in \{ \psi \in \Theta^{\psi} : \| \psi - \psi_0 \| \leq \kappa \}} |R_n (\psi, \hat \gamma^*)| / (1 + \| t_n (\psi) \|)^2 = o_p (1)$ for any sequence $\kappa$ converging to zero.
  This completes the proof of part $(a)$.
  Part $(b)$ follows from the central limit theorem.
  For part $(c)$, the law of large numbers implies that $\mathcal I_n \rightarrow_p \mathcal I := \mathbb E[s_i s_i']$.
  The nonsingularity of $\mathcal I$ follows from the nonsingularity of $\mathbb E [{s^{\eta}_i} {s^{\eta}_i}' ]$ and $\mathbb E [ {s_i^{\sigma\lambda}} {s_i^{\sigma\lambda}}' ]$ by Assumption \ref{covariate}$(b)$ and Assumption \ref{covariate}$(d)$, respectively, and the fact that $\mathbb E[s_i^{\eta} {s_i^{\sigma\lambda}}'] = 0$ because $\mathbb E[H_i^1 H_i^2] = 0$, where $s_i^{\eta} := (U_i H_i^1 / \sigma_0)$ and $s_i^{\sigma\lambda} := (H_i^2 / (2 \sigma_0^2), D_i H_i^2/(8 \sigma_0^2))'$.
\end{proof}

\begin{proof} [Proof of Proposition \ref{asy null dist}]
  Part $(a)$ follows from a simple adaptation of the proof of Proposition 3$(a)$ of \cite{kasahara2015testing} to our quadratic approximation in Proposition \ref{quad approx}$(a)$.
  For part $(b)$, our proof is based on that of Proposition 3$(b)$ and $(c)$ of \cite{kasahara2015testing}.
  We suppress $\psi$ from $t_n (\psi)$, and let $\hat t_n := t_n (\hat \psi^*)$.
  Define $\mathcal I_{\lambda} := \mathbb E[D_i^2 (H_i^2)^2 / (8 \sigma^2_0)^2], \mathcal I_{\eta_{\sigma}\lambda} := (0_{1, q + 1}, \mathbb E [ D_i (H_i^2)^2 / (16 \sigma_0^4)])'$,
  \begin{equation} 
    \mathcal I_{\eta_\sigma} :=
    \begin{bmatrix} 
      \mathbb E [U_i U_i' (H_i^1)^2 / \sigma_0^2] & 0_{q + 1, 1} \\
      0_{1, q + 1} & \mathbb E[(H_i^2)^2 / (2 \sigma_0^2)^2]
    \end{bmatrix},
    \mathcal I := 
    \begin{bmatrix} 
      \mathcal I_{\eta_{\sigma}} & \mathcal I_{\eta_{\sigma}\lambda} \\
      \mathcal I_{\eta_{\sigma}\lambda}' & \mathcal I_{\lambda} \notag
    \end{bmatrix},
  \end{equation}
  where $0_{s,t}$ is a $s \times t$ zero matrix.
  For $\vartheta := (\alpha', \beta, \sigma^2)'$, let $l_{0, n} (\vartheta) := \sum_{i = 1}^n \log g(Y_i | W_i; \vartheta)$ be a log-likelihood function under the null model, where $g(Y|W; \vartheta) := \phi_{\sigma} (Y - X'\alpha - D \beta)$ for a density function $\phi_{\sigma}$ of the normal distribution with mean zero and variance $\sigma^2$.
  Letting $\hat \vartheta$ denote the MLE under the null model, observe that $SLRT = 2(l_n (\hat \psi^*, \hat \gamma^*) - l_n (\psi_0, \hat \gamma^*)) - 2(l_{0, n} (\hat \vartheta) - l_{0, n} (\vartheta_0))$ because $l_n (\psi_0, \hat \gamma) = l_{0, n} (\vartheta_0)$.
  We investigate each of $(i) \ 2 (l_n (\hat \psi^*, \hat \gamma^*) - l_n (\psi_0, \hat \gamma^*))$ and $(ii) \ 2(l_{0, n} (\hat \vartheta) - l_{0, n} (\vartheta_0))$ below.

  \underline{$(i) \ 2 (l_n (\hat \psi^*, \hat \gamma^*) - l_n (\psi_0, \hat \gamma^*))$}.
  Because $l_n (\hat \psi^*, \hat \gamma^*) - l_n (\psi_0, \hat \gamma^*) = l^{\ast}_n (\hat \psi^*, \hat \gamma^*) - l^{\ast}_n (\psi_0, \hat \gamma^*)$, Proposition \ref{quad approx}$(a)$ and part $(a)$ of this proposition yields
  $l_n (\hat \psi^*, \hat \gamma^*) - l_n (\psi_0, \hat \gamma^*) = S_n' \hat t_n - \frac{1}{2} \hat t_n'\mathcal I_n \hat t_n + o_p (1)$.
  Let $W_{\psi} = (W_{\eta_{\sigma}}', W_{\lambda})' := \mathcal I^{-1} S_n$, where $W_{\eta_\sigma}$ is the first $q + 2$ elements of $W_{\psi}$.
  It now follows from $2 S_n' \hat t_n - \hat t_n' \mathcal I \hat t_n = W_{\psi}' \mathcal I W_{\psi} - (\hat t_n - W_{\psi})'\mathcal I (\hat t_n - W_{\psi})$, Proposition \ref{quad approx}$(c)$ and part $(a)$ of this proposition that
  \begin{equation} \label{prop asy null dist i decomp}
    2(l_n (\hat \psi, \hat \gamma) - l_n (\psi_0, \hat \gamma)) = W_{\psi}' \mathcal I W_{\psi} - (\hat t_n - W_{\psi})'\mathcal I (\hat t_n - W_{\psi}) + o_p(1). 
  \end{equation}
  Partition $S_n = (S_{\eta_{\sigma}}', S_{\lambda})'$ with $S_{\eta_{\sigma}}$ being the first $q + 2$ elements of $S_n$. Furthermore, define $\bar W_{\eta_{\sigma}} := \mathcal I_{\eta_{\sigma}}^{-1} S_{\eta_{\sigma}}$ and $\mathcal I_{\eta_{\sigma} \cdot \lambda} := \mathcal I_{\lambda} - \mathcal I_{\eta_{\sigma}\lambda}' \mathcal I_{\eta_{\sigma}}^{-1} \mathcal I_{\eta_{\sigma}\lambda}$.
  By a tedious but straightforward calculation using the formula of inverse of a partitioned matrix for $\mathcal I^{-1}$ (e.g., Exercise 5.16$(a)$ of \cite{abadir2005matrix}), we have 
  \begin{equation} \label{prop asy null dist i inv partition}
    W_{\psi}' \mathcal I W_{\psi} = \bar W_{\eta_{\sigma}}' \mathcal I_{\eta_{\sigma}} \bar W_{\eta_{\sigma}} + W_{\lambda}' \mathcal I_{\eta_{\sigma} \cdot \lambda} W_{\lambda}.
  \end{equation}

  For the second term on the right side of \eqref{prop asy null dist i decomp}, the proof of Theorem 2 of \cite{andrews1999estimation} gives that $(\hat t_n' - W_{\psi})'\mathcal I (\hat t_n - W_{\psi})= \inf_{t \in \Lambda_n} (t - W_{\psi})' \mathcal I (t - W_{\psi}) + o_p(1)$, where $\Lambda_n := \{ t_n (\psi) : \psi \in \Theta^{\psi} \}$.
  By Assumption \ref{parameter compact}$(a)$-$(c)$ and $(e)$, the set $\{ t_n (\psi) / b_n : \psi \in \Theta^{\psi} \}$ is locally approximated by a cone $\Lambda := \mathbb R^{q + 2} \times [0, \infty)$ for any sequence $b_n$ such that $b_n \rightarrow \infty$ and $b_n = o(n^{1/2})$ (see page 1358 of \cite{andrews1999estimation} for the definition of ``locally approximated by a cone'').
  Hence, Lemma 2 of \cite{andrews1999estimation} yields $\inf_{t \in \Lambda_n} (t - W_{\psi})'\mathcal I (t - W_{\psi}) = \inf_{t \in \Lambda} (t - W_{\psi})'\mathcal I (t - W_{\psi}) + o_p(1)$, by which we have 
  \begin{equation} \label{prop asy null dist i locally approx}
    (\hat t_n' - W_{\psi})'\mathcal I (\hat t_n - W_{\psi}) = \inf_{t \in \Lambda} (t - W_{\psi})' \mathcal I (t - W_{\psi}) + o_p(1).
  \end{equation}
  Furthermore, a straightforward calculation with the formula of inverse of a partitioned matrix for $\mathcal I^{-1}$,
  $(t - W_{\psi})' \mathcal I (t - W_{\psi}) = (t_1 + \mathcal I_{\eta_\sigma}^{-1} \mathcal I_{\eta_{\sigma}\lambda} t_2 - \bar W_{\eta_{\sigma}})' \mathcal I_{\eta_{\sigma}} (t_1 + \mathcal I_{\eta_\sigma}^{-1} \mathcal I_{\eta_{\sigma}\lambda} t_2 - \bar W_{\eta_{\sigma}}) + (t_2 - W_{\lambda})'\mathcal I_{\eta_{\sigma} \cdot \lambda} (t_2 - W_{\lambda})$,
  where we partition $t = (t_1', t_2)'$ with $t_1$ being the first $q+2$ elements of $t$.
  Because $t_1$ ranges over $\mathbb R^{q + 2}$ independently of the value of $t_2$, we observe 
  \begin{align} \label{prop asy null dist i inf}
    \inf_{t \in \Lambda} (t - W_{\psi})' \mathcal I (t - W_{\psi}) &= \inf_{t_1 \in \mathbb R^{q + 2}} (t_1 - \bar W_{\eta_\sigma})' \mathcal I_{\eta_{\sigma}} (t_1 - \bar W_{\eta_\sigma}) + \inf_{t_2 \in [0, \infty)} (t_2 - W_{\lambda})' \mathcal I_{\eta_\sigma \cdot \lambda} (t_2 - W_{\lambda}) \notag \\
    &= \inf_{t_1 \in \mathbb R^{q + 2}} (t_1 - \bar W_{\eta_\sigma})' \mathcal I_{\eta_{\sigma}} (t_1 - \bar W_{\eta_\sigma}) + 1 \{ W_\lambda < 0 \} W_\lambda' \mathcal I_{\eta_\sigma \cdot \lambda} W_\lambda.
  \end{align}
  Combining \eqref{prop asy null dist i decomp}, \eqref{prop asy null dist i inv partition}, \eqref{prop asy null dist i locally approx} and \eqref{prop asy null dist i inf} gives
  \begin{align} 
    &2(l_n (\hat \psi, \hat \gamma) - l_n (\psi_0, \hat \gamma)) \notag \\
    = \ &\bar W_{\eta_\sigma}' \mathcal I_{\eta_{\sigma}} \bar W_{\eta_{\sigma}} + 1 \{ W_{\lambda} \geq 0 \} W_{\lambda}' \mathcal I_{\eta_\sigma \cdot \lambda} W_{\lambda} - \inf_{t_1 \in \mathbb R^{q + 2}} (t_1 - \bar W_{\eta_\sigma})' \mathcal I_{\eta_{\sigma}} (t_1 - \bar W_{\eta_\sigma}) + o_p(1) \notag
  \end{align}

  \underline{$(ii) \ 2 (l_{0, n} (\hat \vartheta) - l_{0, n} (\vartheta_0))$}.
  Define $u_n (\vartheta) := (n^{1/2} (\alpha - \alpha_0)', n^{1/2} (\beta - \beta_0), n^{1/2} (\sigma^2 - \sigma^2_0))'$
  and $s_i^{\eta_\sigma} := (U_i' H_i^1 / \sigma_0, H_i^2 / (2 \sigma_0^2))'$. By a similar argument to the proof of Proposition \ref{quad approx}, we obtain the following quadratic approximation.
  \begin{equation} 
    l_{0, n} (\vartheta) - l_{0, n} (\vartheta_0) = {S_n^{\eta_\sigma}}'u_n (\vartheta) - \frac{1}{2} u_n (\vartheta)' \mathcal I^{\eta_\sigma}_n u_n (\vartheta) + R_n (\vartheta), \notag
  \end{equation}
  where $S_n^{\eta_\sigma} := n^{-1/2} \sum_{i = 1}^n s_i^{\eta_\sigma}$ and $\mathcal I_n^{\eta_\sigma} := n^{-1} \sum_{i = 1}^n s_i^{\eta_\sigma} {s_i^{\eta_\sigma}}'$, and $\sup_{\vartheta: \| \vartheta - \vartheta_0 \| \leq \kappa} |R_n (\vartheta)| / (1 + \| u_n (\vartheta) \|)^2 = o_p(1)$ for any $\kappa$ converging to zero.
  Then, similarly to part $(a)$ of this proposition, we can show $u_n (\hat \vartheta) = O_p(1)$.
  In view of the above quadratic approximation and $u_n (\hat \vartheta) = O_p(1)$, repeating the argument for part $(i)$ gives that 
  \begin{equation} 
    2 (l_{0, n} (\hat \vartheta) - l_{0, n} (\vartheta_0)) = \bar W_{\eta_\sigma}' \mathcal I_{\eta_\sigma} \bar W_{\eta_\sigma} - \inf_{t_1 \in \mathbb R^{q + 2}} (t_1 - \bar W_{\eta_\sigma})' \mathcal I_{\eta_\sigma} (t_1 - \bar W_{\eta_\sigma}) + o_p(1). \notag
  \end{equation}
  
  Consequently, combining the results from $(i)$ and $(ii)$ yields 
  $SLRT = 1 \{ W_{\lambda} \geq 0 \} W_{\lambda}' \mathcal I_{\eta_\sigma \cdot \lambda} W_{\lambda} + o_p(1) = 1 \{ \mathcal I_{\eta_\sigma \cdot \lambda}^{1/2} W_{\lambda} \geq 0 \} ( \mathcal I_{\eta_\sigma \cdot \lambda}^{1/2} W_{\lambda} )^2 + o_p(1)$ from $\mathcal I_{\eta_\sigma \cdot \lambda} > 0$.
  By the central limit theorem and Slutsky's theorem, $W_{\psi} \rightarrow_d N(0, \mathcal I^{-1})$.
  In particular, because $W_{\lambda}$ is the last coordinate of $W_{\psi}$ and $\mathcal I_{\eta_{\sigma} \cdot \lambda}^{-1}$ is the bottom right element of $\mathcal I^{-1}$, $W_{\lambda} \rightarrow_d N(0, \mathcal I_{\eta_\sigma \cdot \lambda}^{-1})$.
  Hence, $\mathcal I_{\eta_\sigma \cdot \lambda}^{1/2} W_{\lambda} \rightarrow_d N(0, 1)$.
  Because the map $x \rightarrow 1 \{ x \geq 0 \} x^2$ is continuous almost everywhere with respect to Lebesgue measure, the continuous mapping theorem completes the proof.
\end{proof}
\section{Auxiliary results}

Lemma \ref{compact} compactifies the parameter space $\Theta \times \Gamma$ in Assumption \ref{parameter}, which is helpful for the proofs in Appendix A.

\begin{lemma}\label{compact}
  Assume Assumptions \ref{parameter} and \ref{covariate} hold. Then $(\hat \theta^*, \hat \gamma^*) = \argmax_{\theta \in \tilde \Theta, \gamma \in \Gamma_M} l^*_n (\theta, \gamma)$ with probability approaching one, where $\tilde \Theta := \Theta^{\alpha} \times \Theta^{\beta} \times \Theta^{\lambda} \times \tilde \Theta^{\sigma^2}$ with $ \tilde \Theta^{\sigma^2} := [l_{\sigma^2_0}, u_{\sigma^2_0}]$ for some $0 < l_{\sigma^2_0} < \sigma^2_0 <  u_{\sigma^2_0} < \infty$
  and $\Gamma_M := \{ \gamma \in \Gamma : \| \gamma \|_1 \leq M n/ p \}$ for some finite $M > 0$.
\end{lemma}

\begin{proof}[Proof of Lemma \ref{compact}]
  We first prove $(\hat \theta^*, \hat \gamma^*) = \argmax_{\theta \in \tilde \Theta, \gamma \in \Gamma} l_n^{\ast} (\theta, \gamma)$ with probability approaching one.
  Our argument is based on Lemma 3.1 of \cite{chen2017consistency}.
  By a straightforward calculation, 
  \begin{equation} \label{density bound}
    \log f(y | w; \theta, \gamma) \leq - \frac{\log (2 \pi)}{2} - \frac{\log \sigma^2}{2} - \frac{ (y - x'\alpha - d(\beta + \lambda))^2 \land (y - x'\alpha - d\beta)^2 }{2 \sigma^2},
  \end{equation}
  which implies that $\sup_{\alpha \in \Theta^{\alpha}, \beta \in \Theta^{\beta}, \lambda \in \Theta^{\lambda}, \gamma \in \Gamma} n^{-1} l_n^{\ast} (\theta, \gamma) - n^{-1} l^{\ast}_n (\theta_0, 0)$ is bounded by
  \begin{equation}\label{compact ll bound}
    - \frac{\log(2 \pi) + \log \sigma^2}{2} - \frac{1}{2 \sigma^2} \inf_{\alpha \in \Theta^{\alpha}, \beta \in \Theta^{\beta}, \lambda \in \Theta^{\lambda}} \mathbb P_n (Y - X'\alpha - D(\beta + \lambda))^2 \land (Y - X'\alpha - D\beta)^2 
    - \frac{l_n^{\ast} (\theta_0, 0)}{n}.
  \end{equation}

  Let $\mathcal A_n = \| (\mathbb P_n - P) (Y - X'\alpha - D(\beta + \lambda))^2 \land (Y - X'\alpha - D\beta)^2 \|_{\Theta^{\alpha} \times \Theta^{\beta} \times \Theta^{\lambda}}$. 
  Then, for \eqref{compact ll bound}, $\inf_{\alpha \in \Theta^{\alpha}, \beta \in \Theta^{\beta}, \lambda \in \Theta^{\lambda}} \mathbb P_n (Y - X'\alpha - D(\beta + \lambda))^2 \land (Y - X'\alpha - D\beta)^2$ is no smaller than
  \begin{equation}\label{inf A bound}
    - \mathcal A_n + \inf_{\alpha \in \Theta^{\alpha}, \beta \in \Theta^{\beta}, \lambda \in \Theta^{\lambda}} P (Y - X'\alpha - D(\beta + \lambda))^2 \land (Y - X'\alpha - D\beta)^2.
  \end{equation}
  We consider bounding the right side from below.
  For $\mathcal A_n$, it follows from Lemma 2.6.15, Lemma 2.6.18(v), Theorem 2.6.7 and Theorem 2.4.3 of \cite{van1996weak} and Assumption \ref{covariate}$(a)$ that each of $\{ Y_i - X_i'\alpha - D_i(\beta + \lambda) : \alpha \in \Theta^{\alpha}, \beta \in \Theta^{\beta}, \lambda \in \Theta^{\lambda} \}$ and $\{ Y_i - X_i'\alpha - D_i\beta : \alpha \in \Theta^{\alpha}, \beta \in \Theta^{\beta} \}$ is 
  a Glivenko-Cantelli class. Then, by Theorem 3 of \cite{van2000preservation} and Assumption \ref{covariate}$(a)$, $\mathcal A_n \rightarrow_p 0$.
  For $\inf_{\alpha \in \Theta^{\alpha}, \beta \in \Theta^{\beta}, \lambda \in \Theta^{\lambda}} P (Y - X'\alpha - D(\beta + \lambda))^2 \land (Y - X'\alpha - D\beta)^2$, 
  note that $Y - X'\alpha - D(\beta + \lambda) = \varepsilon + X'(\alpha_0 - \alpha) + D(\beta_0 - \beta + \delta \lambda_0 - \lambda)$.
  Then $\mathbb P (Y - X'\alpha - D(\beta + \lambda) = 0) = \mathbb E[\mathbb P(\ve + X'(\alpha_0 - \alpha) + D(\beta_0 - \beta + \delta \lambda_0 - \lambda) = 0 | X, D, \delta)]$.
  Because $\ve$ and $(X, D, \delta)$ are independent, the conditional probability inside the expectation on the right side is zero so that $\mathbb P (Y - X'\alpha - D(\beta + \lambda) = 0) = 0$.
  Similarly, $\mathbb P(Y - X'\alpha - D\beta = 0) = 0$.
  Hence, $P (Y - X'\alpha - D(\beta + \lambda))^2 \land (Y - X'\alpha - D\beta)^2 > 0$ over $\Theta^\alpha \times \Theta^{\beta} \times \Theta^{\lambda}$.
  Because $ P(Y - X'\alpha - D(\beta + \lambda))^2 \land (Y - X'\alpha - D\beta)^2$ is continuous in $(\alpha', \beta, \lambda)'$ from Assumption \ref{covariate}$(a)$ and the dominated convergence theorem, and $\Theta^{\alpha} \times \Theta^{\beta} \times \Theta^{\lambda}$ is compact,
  it holds that $\inf_{\alpha \in \Theta^{\alpha}, \beta \in \Theta^{\beta}, \lambda \in \Theta^{\lambda}} P (Y - X'\alpha - D(\beta + \lambda))^2 \land (Y - X'\alpha - D\beta)^2 > 0$.
  Combining this inequality with $\mathcal A_n \rightarrow_p 0$ and \eqref{inf A bound} yields that there exists a finite positive constant $M_1$ such that 
  \begin{equation} \label{inf lower bound}
    \inf_{\alpha \in \Theta^{\alpha}, \beta \in \Theta^{\beta}, \lambda \in \Theta^{\lambda}} \mathbb P_n (Y - X'\alpha - D(\beta + \lambda))^2 \land (Y - X'\alpha - D\beta)^2 > M_1
  \end{equation}
  with probability approaching one.

  For $n^{-1} l^{*}_n (\theta_0, 0)$ on the right side of \eqref{compact ll bound}, we first note the following inequality: for any positive real numbers $a$ and $b$,
  $ |\log (a/2 + b/2)| \leq |\log (a \land b)| \lor |\log (a \lor b)| = |\log a| \lor |\log b| \leq |\log a| + |\log b|$.
  Applying this inequality to $|\log f(Y | W; \theta_0, 0)|$, $P |\log f (Y| W; \theta_0, 0)|$ is bounded by
  \begin{equation} 
    \left| \log \sigma_0 \sqrt{2 \pi} \right| + \frac{1}{4 \sigma_0^2} \left( P (Y - X' \alpha_0 - D (\beta_0 + \lambda_0))^2 + P (Y - X' \alpha_0 - D \beta_0 )^2 \right), \notag
  \end{equation}
  which is finite by Assumption \ref{covariate}$(a)$. Hence, by the law of large numbers, there exists a finite positive constant $M_2$ such that $|n^{-1} l^{*}_n (\theta_0, 0)| \leq M_2$ holds with probability approaching one.

  In view of this bound, \eqref{compact ll bound} and \eqref{inf lower bound}, 
  \begin{equation} 
    \sup_{\alpha \in \Theta^{\alpha}, \beta \in \Theta^{\beta}, \lambda \in \Theta^{\lambda}, \gamma \in \Gamma} n^{-1} l_n^{\ast} (\theta, \gamma) - n^{-1} l^{\ast}_n (\theta_0, 0) \leq - \frac{\log(2 \pi)}{2} - \frac{\log \sigma^2}{2} -  \frac{M_1}{2 \sigma^2} + M_2 \notag
  \end{equation}
  holds for any $\sigma^2$ with probability approaching one. Because the right side tends to minus infinity as $\sigma^2 \rightarrow 0$ or $\sigma^2 \rightarrow \infty$, $\sup_{\alpha \in \Theta^{\alpha}, \beta \in \Theta^{\beta}, \lambda \in \Theta^{\lambda}, \sigma^2 \in (0, l) \cup (u, \infty ), \gamma \in \Gamma} n^{-1} l^{\ast}_n (\theta, \gamma) - n^{-1} l^{\ast}_n (\theta_0, 0) < 0$ holds with probability approaching one
  for some $0 < l < \sigma^2_0 < u < \infty$.
  This proves $(\hat \theta^*, \hat \gamma^*) = \argmax_{\theta \in \tilde \Theta, \gamma \in \Gamma} l_n^{\ast} (\theta, \gamma)$ with probability approaching one.

  We move on to verify $(\hat \theta^*, \hat \gamma^*) = \argmax_{\theta \in \widetilde \Theta, \gamma \in \Gamma_{M}} l^{\ast}_n (\theta, \gamma)$ with probability approaching one.
  Define $(\tilde \theta^*, \tilde \gamma^*) := \argmax_{\theta \in \tilde \Theta, \gamma \in \Gamma} l^{*}_n (\theta, \gamma)$.
  Then $\mathbb P(n^{-1} l^{*}_n (\tilde \theta^*, \tilde \gamma^*) \geq -M_2) \geq \mathbb P(n^{-1} l^{\ast}_n (\theta_0, 0) \geq -M_2) \rightarrow 1$ because $|n^{-1} l_n^* (\theta_0, 0)| \leq M_2$ with probability approaching one.
  Furthermore, \eqref{density bound} implies that $n^{-1} l_n (\tilde \theta^*, \tilde \gamma^*) \leq -\log(2 \pi) / 2 - \log (l_{\sigma_0^2}) / 2$. 
  Consequently,
  \begin{align} 
    &\mathbb P\left(n^{-1} l^{\ast}_n (\tilde \theta^*, \tilde \gamma^*) \geq -M_2, n^{-1} l_n (\tilde \theta^*, \tilde \gamma^*) \leq -\log(2 \pi) / 2 - \log (l_{\sigma_0^2}) / 2\right) \notag \\
    \leq \ &\mathbb P \left( p/n \| \tilde \gamma^* \|_1 \leq -\log(2 \pi)/2 - \log (l_{\sigma_0^2}) / 2 + M_2 \right). \notag 
  \end{align}
  Because the left side converges to one, the desired result follows by setting $M = -\log (2 \pi) / 2 - \log (l_{\sigma_0^2})/2 + M_2$.
\end{proof}

The following lemma is the key to proving the consistency and the convergence rate in Proposition \ref{consistency}.
This result is an adaptation of Lemma A1 of \cite{andrews1993tests} to our high-dimensional setting.
\begin{lemma} \label{consistency lemma}
  Assume the assumption of Proposition \ref{consistency} holds.
  Let $\{ c_n \}_{n \in \mathbb N}$ be a sequence of positive real numbers converging to zero and $a_n := \mathbb E [\sup_{\theta \in \tilde \Theta, \gamma \in \Gamma_M} |n^{-1} l_n^{\ast} (\theta, \gamma) - \mathbb E[\log f(Y | W ; \theta, \gamma)] + p_n / n \|\gamma \|_1|]$.
  For $\ve > 0$, define $b_{\ve, n} = \mathbb E[\log f(Y | W; \theta_0, 0)] - \sup_{(\theta', \gamma')' \in \Xi_{\ve, n}} (\mathbb E[\log f(Y | W; \theta, \gamma)] - p/n \| \gamma\|_1)$
  with $\Xi_{\ve, n} := \{ (\theta, \gamma) \in \tilde \Theta \times \Gamma_M : \| \theta - \theta_0 \| + c_n^{-1} \| \gamma \|_1 \geq \ve \}$.
  Then if $a_n = o(b_{\ve, n})$ for each $\ve > 0$, it holds that $\hat \theta^* \rightarrow_p \theta_0$ and $\| \hat \gamma^* \|_1 = o_p (c_n)$.
\end{lemma}

\begin{proof} [Proof of Lemma \ref{consistency lemma}]
  The proof is based on Lemma A1 of \cite{andrews1993tests}.
  By the assumption on $b_{\ve, n}$, 
  \begin{equation} \label{consistency lemma prob bound}
    \mathbb P ((\hat \theta^*, \hat \gamma^*) \in \Xi_{\ve, n}) 
    \leq
    \mathbb P \left( \mathbb E[\log f(Y | W; \theta_0, 0)] - (\mathbb E[\log f(Y | W; \hat \theta^*, \hat \gamma^*)] - p_n / n \| \hat \gamma^* \|_1) \geq b_{\ve, n} \right).
  \end{equation}
  Now, for the term inside the probability on the right side, $\mathbb E[\log f(Y | W; \theta_0, 0)] - (\mathbb E[\log f(Y | W; \hat \theta^*, \hat \gamma^*)] - p_n / n \| \hat \gamma^* \|_1)$ is bounded by
  \begin{align} 
    & \mathbb E[\log f(Y | W; \theta_0, 0)] - n^{-1} l^{\ast}_n (\hat \theta^*, \hat \gamma^*) + n^{-1} l^{\ast}_n (\hat \theta^*, \hat \gamma^*) - \mathbb E[\log f(Y | W; \hat \theta^*, \hat \gamma^*)] + p_n / n \| \hat \gamma^* \|_1 \notag \\
    \leq \ & \mathbb E[\log f(Y | W; \theta_0, 0)] - n^{-1} l^{\ast}_n (\theta_0, 0) + n^{-1} l^{\ast}_n (\hat \theta^*, \hat \gamma^*) - \mathbb E[\log f(Y | W; \hat \theta^*, \hat \gamma^*)] + p_n / n \| \hat \gamma^* \|_1 \notag \\
    \leq \ & 2 \sup_{\theta \in \tilde \Theta, \gamma \in \Gamma_M} \left| n^{-1} l^{\ast}_n (\theta, \gamma) - \mathbb E[\log f(Y|W; \theta, \gamma)] + p_n / n \| \gamma \|_1 \right|, \notag
  \end{align}
  where the first inequality follows from the definition of $(\hat \theta^*, \hat \gamma^*)$.
  Combining this inequality with \eqref{consistency lemma prob bound} and Markov's inequality gives $\mathbb P ((\hat \theta^*, \hat \gamma^*) \in \Xi_{\ve, n}) \leq a_n / b_{\ve, n} = o(1)$.
  This completes the proof.
\end{proof}

Lemma \ref{multivariate contraction} is multivariate contraction principle, which is instrumental in handling the high-dimensionality in the proof of Proposition \ref{consistency}.
Similar but slightly different results are obtained in Theorem 4.1 of \cite{van2013generic} and Theorem 16.2 of \cite{van2016estimation}.
\begin{lemma} \label{multivariate contraction}
  Let $\{ X_i \}_{i = 1}^n$ be $\mathcal X$-valued $i.i.d.$ random variables for some measurable space $(\mathcal X, \mathcal S)$
  and $\mathcal F$ be a class of $\mathbb R^{r}$-valued measurable functions on $\mathcal X$.
  Consider $L_1$-Lipschitz functions $\rho_i: \mathbb R^r \rightarrow \mathbb R$ such that $|\rho_i(z) - \rho_i(\tilde z)| \leq \| z - \tilde z \|_1$ for all $z, \tilde z \in \mathbb R^{r}$ and $i = 1, \dots, n$.
  Let $\{ \xi_i \}_{i = 1}^n$ be $i.i.d.$ Rademacher random variables and $\{ \omega_{i, k} : 1 \leq i \leq n, 1 \leq k \leq r \}$ be a collection of $i.i.d.$ standard normal random variables, both of which are independent of each other and of $\{ X_i \}_{i = 1}^n$.
  Then it holds that 
  \begin{equation} 
    \mathbb E\left[ \sup_{f, g \in \mathcal F} \left| \sum_{i = 1}^n \xi_i (\rho_i (f (X_i)) - \rho_i ( g(X_i))) \right| \right] \lesssim \mathbb E \left[ \sup_{f \in \mathcal F} \sum_{k = 1}^r \sum_{i = 1}^n \omega_{i, k} f_k (X_i) \right], \notag
  \end{equation}
  with $f := (f_1, \dots, f_r)'$.
\end{lemma}

\begin{proof} [Proof of Lemma \ref{multivariate contraction}]
  We follow the proof of Theorem 4.1 of \cite{van2013generic} and that of Theorem 16.2 of \cite{van2016estimation}.
  Observe that 
  \begin{equation} 
    \mathbb E\left[ \sup_{f, g \in \mathcal F} \left| \sum_{i = 1}^n \xi_i (\rho_i (f (X_i)) - \rho_i (g(X_i))) \right| \right] 
    = \mathbb E \left[ \mathbb E \left[ \sup_{f, g \in \mathcal F} \left| \sum_{i = 1}^n \xi_i (\rho_i (f (X_i)) - \rho_i (g(X_i))) \right| \middle| X^{(n)} \right] \right],
  \end{equation}
  where $X^{(n)} := (X_1, \dots, X_n)'$.
  We investigate the tail behavior of a centered, symmetric stochastic process $(\sum_{i = 1}^n \xi_i \rho_i (f(X_i)))_{f \in \mathcal F}$ with $X^{(n)}$ fixed.
  Note that, for any $f, g \in \mathcal F$, 
  \begin{equation}\label{multivariate contraction cs} 
    \sum_{i = 1}^n (\rho_i (f(X_i)) - \rho_i (g(X_i)))^2 \leq \sum_{i = 1}^n \| f(X_i) - g(X_i) \|_1^2
    \leq r \sum_{i = 1}^n \sum_{k = 1}^r (f_k (X_i) - g_k (X_i))^2,
  \end{equation}
  where $g = (g_1, \dots, g_r)'$ and the last inequality follows from the Cauchy-Schwarz inequality for $\| f(X_i) - g(X_i) \|_1$.
  For $u > 0$ and $(\rho_1 (f(X_i)), \dots, \rho_n (f(X_n)))' \neq (\rho_1 (g(X_1)), \dots, \rho_n (g(X_n)))'$, 
  Lemma 2.2.7 of \cite{van1996weak} yields that
  \begin{equation} \label{multivariate contraction hoeffding}
    \mathbb P \left( \left| \sum_{i = 1}^n \xi_i (\rho_i (f(X_i)) - \rho_i (g(X_i))) \right| \geq u \middle| X^{(n)} \right)
    \leq 2 \exp \left\{ - \frac{u^2}{2} \left( \sum_{i = 1}^n (\rho_i (f(X_i)) - \rho_i (g(X_i)))^2 \right)^{-1} \right\}.
  \end{equation}
  It now follows from \eqref{multivariate contraction cs} and \eqref{multivariate contraction hoeffding} that 
  \begin{equation} \label{multivariate contraction tail bound}
    \mathbb P \left( \left| r^{-1/2} \sum_{i = 1}^n \xi_i (\rho_i (f(X_i)) - \rho_i (g(X_i))) \right| \geq u \middle| X^{(n)} \right)
    \leq 2 \exp \left\{ - \frac{u^2}{2} \left( \sum_{i = 1}^n \sum_{k = 1}^r (f_k (X_i) - g_k (X_i))^2 \right)^{-1} \right\}.
  \end{equation}

  Let $e (f, g) := \left( \sum_{i = 1}^n \sum_{k = 1}^r (f_k (X_i) - g_k (X_i))^2 \right)^{1/2}$. This $e$ is the canonical semi-metric as in (2.113) of \cite{talagrand2021upper} for 
  a centered Gaussian process $( \sum_{k = 1}^r \sum_{i = 1}^n \omega_{i, k} f_k (X_i) )_{f \in \mathcal F}$ with $X^{(n)}$ fixed because,
  for any $f, g \in \mathcal F$, 
  \begin{equation}
    \mathbb E \left[ \left( \sum_{k = 1}^r \sum_{i = 1}^n \omega_{i, k} (f_k(X_i) - g_k (X_i)) \right)^2 \middle| X^{(n)} \right] = \sum_{k = 1}^r \sum_{i = 1}^n (f_k (X_i) - g_k (X_i))^2, \notag
  \end{equation}
  by the assumption on $\{ \omega_{i, k} : 1 \leq i \leq n, 1 \leq k \leq r \}$.
  In view of \eqref{multivariate contraction tail bound}, Theorem 2.10.11 of \cite{talagrand2021upper} gives that 
  \begin{equation} 
    \mathbb E \left[ \sup_{f, g \in \mathcal F} \left| \sum_{i = 1}^n \xi_i (\rho_i (f(X_i)) - \rho_i (g(X_i))) \right| \middle| X^{(n)} \right]
    \lesssim r^{1/2} \mathbb E \left[ \sup_{f \in \mathcal F} \sum_{k = 1}^r \sum_{i = 1}^n \omega_{i, k} f_k (X_i) \middle| X^{(n)} \right]. \notag
  \end{equation}
  Taking the expectation with respect to $X^{(n)}$ completes the proof.
\end{proof}

The following lemma plays an important role in the proof of Proposition \ref{consistency}.
\begin{lemma} \label{maximal}
  Assume the assumption of Proposition \ref{consistency} holds. Then it holds that \\
  $(a) \ \mathbb E \left[ \left\| \mathbb P_n \omega Z'\gamma \right\|_{\Gamma_M} \right] \lesssim \sqrt{n \log d} / p$, \\
  $(b) \ \mathbb E\left[ \left\| \mathbb P_n \omega \frac{(Y - X'\alpha - D(\beta + \lambda))^2}{2 \sigma^2} \right\|_{\tilde \Theta} \right] \lesssim n^{-1/2}$, \\
  $(c) \ \mathbb E \left[ \left\| \mathbb P_n \omega \frac{(Y - X'\alpha - D\beta)^2}{2 \sigma^2} \right\|_{\tilde \Theta} \right] \lesssim n^{-1/2}$.
\end{lemma}

\begin{proof} [Proof of Lemma \ref{maximal}]
  \underline{$(a)$}.
  Observe that 
  \begin{equation} \label{lemma maximal a prod bound}
    \mathbb E \left[ \left\| \mathbb P_n \omega Z'\gamma \right\|_{\Gamma_M} \right] = \mathbb E \left[ \left\| \left( \frac{1}{n} \sum_{i = 1}^n \omega_i Z_i \right)' \gamma  \right\|_{\Gamma_M} \right]
    \leq \mathbb E \left[ \max_{1 \leq j \leq d} \left| \mathbb P_n \omega Z_{(j)} \right| \right] \sup_{\gamma \in \Gamma_M} \| \gamma \|_1.
  \end{equation}
  By Lemma 8 of \cite{chernozhukov2015comparison},
  \begin{equation} \label{lemma maximal cck bound}
    \mathbb E \left[ \max_{1 \leq j \leq d} \left| \mathbb P_n \omega Z_{(j)} \right| \right]
    \lesssim \frac{1}{\sqrt{n}} \left( \sqrt{\max_{1 \leq j \leq d} \frac{1}{n} \sum_{i = 1}^n \mathbb E \left[Z_{(j), i}^2 \right] } \sqrt{\log d} + 
    \frac{1}{\sqrt{n}} \sqrt{\mathbb E \left[\max_{1 \leq i \leq n, 1 \leq j \leq d} \left( \omega_i Z_{(j), i}\right)^2 \right]} \log d \right).
  \end{equation}
  For the second term on the right side, the independence of $\{ \omega_{i} \}_{i = 1}^n$ and $\{ Z_{(j), i} \}_{1 \leq i \leq n, 1 \leq j \leq d}$ implies that 
  $\mathbb E \left[ \max_{1 \leq i \leq n, 1 \leq j \leq d} (\omega_i Z_{(j), i})^2 \right] \leq \mathbb E \left[ \max_{1 \leq i \leq n} \omega_i^2 \right] \mathbb E \left[ \max_{1 \leq i \leq n, 1 \leq j \leq d} Z_{(j), i}^2 \right]$.
  The right side is further bounded by $\left\| \max_{1 \leq i \leq n} | \omega_i | \right\|_{\psi_2}^2 \| \max_{1 \leq i \leq n, 1 \leq j \leq d} |Z_{(j), i}| \|_{\psi_2}^2$ by page 95 of \cite{van1996weak}.
  Here, $\| \cdot \|_{\psi_2}$ is the Orlicz norm for a function $\psi_2 (x) = e^{x^2} - 1$ as defined on page 95 of \cite{van1996weak}.
  Then, by Lemma 2.2.1 and 2.2.2 of \cite{van1996weak} in conjunction with Assumption \ref{covariate}$(c)$,
  $ \left\| \max_{1 \leq i \leq n} | \omega_i | \right\|_{\psi_2}^2 \| \max_{1 \leq i \leq n, 1 \leq j \leq d} |Z_{(j), i}| \|_{\psi_2}^2 \lesssim \log (n + 1) \log(nd + 1)$.
  Consequently, we obtain 
  \begin{equation} \label{lemma maximal square bound}
    \sqrt{\mathbb E \left[\max_{1 \leq i \leq n, 1 \leq j \leq d} \left( \omega_i Z_{(j), i}\right)^2 \right]} \lesssim \sqrt{\log n} \sqrt{\log(nd)}.
  \end{equation}
  Because $\mathbb E\left[Z_{(j), i}^2 \right]$ is bounded uniformly in $i$ and $j$ from Lemma 2.2.1 of \cite{van1996weak} and Assumption \ref{covariate}$(c)$, it follows from
  \eqref{lemma maximal cck bound} and \eqref{lemma maximal square bound} that
  \begin{equation} \label{lemma maximal log bound}
    \mathbb E \left[ \max_{1 \leq j \leq d} \left| \mathbb P_n \omega Z_{(j)} \right| \right] \lesssim \sqrt{\frac{\log d}{n}} + \frac{\sqrt{\log n} \sqrt{\log (n \lor d)} \log d}{n} \lesssim \sqrt{\frac{\log d}{n}},
  \end{equation}
  where the last inequality follows from Assumption \ref{tuning parameter}$(b)$.
  We complete the proof by \eqref{lemma maximal a prod bound} and \eqref{lemma maximal log bound} in conjunction with $\sup_{\gamma \in \Gamma_M} \| \gamma \|_1 \lesssim n/p$ from Assumption \ref{parameter compact}$(d)$.

  \underline{$(b)$}.
  Observe that, by the assumption on $\omega$,
  \begin{equation} 
    \mathbb E\left[ \left\| \mathbb P_n \omega \frac{(Y - X'\alpha - D(\beta + \lambda))^2}{2 \sigma^2} \right\|_{\tilde \Theta} \right] = \mathbb E\left[ \left\| (\mathbb P_n - P) \omega \frac{(Y - X'\alpha - D(\beta + \lambda))^2}{2 \sigma^2} \right\|_{\tilde \Theta} \right]. \label{lemma maximal b simple bound}
  \end{equation}
  Let $\theta_1 = (\alpha_1', \beta_1, \lambda_1, \sigma^2_1)'$ and $\theta_2 = (\alpha_2', \beta_2, \lambda_2, \sigma^2_2)' \in \tilde \Theta$ be arbitrary.
  By the mean value theorem and the Cauchy-Schwarz inequality, 
  \begin{align} 
    &\left| \omega \frac{(Y - X'\alpha_1 - D (\beta_1 + \lambda_1))}{2 \sigma^2_1} - \omega \frac{(Y - X'\alpha_2 - D (\beta_2 + \lambda_2))}{2 \sigma^2_2}  \right| \notag \\
    \leq & \left\| \nabla_{\theta} \omega \frac{(Y - X' \bar \alpha - D (\bar \beta + \bar \lambda))}{2 \bar \sigma^2} \right\| \| \theta_1 - \theta_2 \|, \label{lemma maximal b cs}
  \end{align}
  where $\bar \theta$ lies on the path connecting $\theta_1$ and $\theta_2$.
  By a straightforward derivative calculation in conjunction with Assumption \ref{covariate}$(d)$ and \ref{parameter compact}, $\left\| \nabla_{\theta} \omega \frac{(Y - X' \bar \alpha - D (\bar \beta + \bar \lambda))}{2 \bar \sigma^2 } \right\|$ is bounded by $F := \mathcal C |\omega| (|Y| + \| X \| + 1)^2$.
  Taking $\mathcal C$ sufficiently large, this $F$ can be an envelope function (see page 84 of \cite{van1996weak} for the definition) for the functional class $\mathcal F := \left\{ \omega \frac{(Y - X'\alpha - D (\beta + \lambda))}{2 \sigma^2} : \theta \in \tilde \Theta \right\}$ by Assumption \ref{parameter compact}.
  Then, in view of \eqref{lemma maximal b cs}, Lemma 26 of \cite{katoep} yields $\sup_{Q} N (\ve \| F \|_{Q, 2}, \mathcal F, L_2 (Q) ) \leq (A / \ve)^{\nu}$ for all $0 < \ve < 1$
  with some $A \geq e$ and $\nu \geq 1$, where $N(\cdot, \cdot, \cdot)$ is a covering number (see page 84 of \cite{van1996weak} for the definition) and
  the supremum is taken over all discrete probability measures.
  Because $\mathbb E[F^2]$ is finite from Assumption \ref{covariate}$(a)$, it follows from Corollary 5.1 of \cite{chernozhukov2014gaussian} that
  \begin{equation} 
    \mathbb E\left[ \left\| (\mathbb P_n - P) \omega \frac{(Y - X'\alpha - D(\beta + \lambda))^2}{2 \sigma^2} \right\|_{\tilde \Theta} \right] \lesssim \frac{1}{\sqrt{n}} \left(1 + \frac{  \sqrt{ \mathbb E \left[\max_{1 \leq i \leq n} F_i^2\right] }}{\sqrt{n}} \right) \lesssim n^{-1/2}, \label{lemma maximal b cck bound}
  \end{equation}
  where the second inequality follows from the fact that $\sqrt{ \mathbb E \left[ \max_{1 \leq i \leq n} F_i^2\right]} \leq \sqrt{n} \sqrt{\mathbb E[F^2]}$.
  \eqref{lemma maximal b simple bound} and \eqref{lemma maximal b cck bound} now complete the proof.

  \underline{$(c)$}.
  The proof is similar to that of $(b)$ and thus omitted.
\end{proof}

Lemma \ref{hermite} provides a simplified form of derivatives of the density function for the normal distribution and is cited multiple times in the proof of Proposition \ref{quad approx}.
This result is essentially due to Proposition A of \cite{kasahara2015testing}.
\begin{lemma} \label{hermite}
  Let $\eta = (\eta_1, \dots, \eta_{q + 1})'$ and $U = (U_{(1)}, \dots, U_{(q + 1)})'$. Then the following equalities hold for any nonnegative integer $k_1, \dots, k_q, k_{\lambda}$ and $l$:
  \begin{align} 
    &\nabla_{\eta_1}^{k_1} \dots \nabla_{\eta_q}^{k_q} \nabla_{\lambda}^{k_{\lambda}} \nabla_{\sigma^2}^l \phi_{\sigma} (Y - U'\eta -  D\lambda/2) \notag \\
    = &\left(\prod_{j = 1}^{q} U_{(j)}^{k_j}\right) \left( \frac{D}{2} \right)^{k_{\lambda}} \left( \frac{1}{2} \right)^{l} \left( \frac{1}{\sigma} \right)^{k + 2 l} H^{k + 2 l} \left( \frac{Y - U'\eta - D \lambda / 2}{\sigma} \right) \phi_{\sigma} (Y - U'\eta - D \lambda /2) \notag \\
    &\nabla_{\eta_1}^{k_1} \dots \nabla_{\eta_q}^{k_q} \nabla_{\lambda}^{k_{\lambda}} \nabla_{\sigma^2}^l \phi_{\sigma} (Y - U'\eta +  D\lambda/2) \notag \\
    = &\left(\prod_{j = 1}^{q} U_{(j)}^{k_j}\right) \left( \frac{-D}{2} \right)^{k_{\lambda}} \left( \frac{1}{2} \right)^{l} \left( \frac{1}{\sigma} \right)^{k + 2 l} H^{k + 2 l} \left( \frac{Y - U'\eta + D \lambda / 2}{\sigma} \right) \phi_{\sigma} (Y - U'\eta + D \lambda /2), \notag
  \end{align}
  where $k := k_1 + \dots + k_q + k_{\lambda}$.
\end{lemma}
\begin{proof} [Proof of Lemma \ref{hermite}]
  The statement follows from a minor modification of the proof of Proposition A of \cite{kasahara2015testing}.
\end{proof}

This lemma is the key to showing that the effect of $\hat \gamma^*$ on quadratic approximation for the penalized log-likelihood function vanishes asymptotically in the proof of Proposition \ref{quad approx}.
\begin{lemma} \label{shrinkage pi}
  Assume the assumption of Proposition \ref{quad approx} holds. Let $\{ V_i \}_{i = 1}^n $ be $i.i.d.$ random variables with finite second moment and $\rho (\ve_i)$ be a polynomial of $\ve_i$.
  Suppose that $\{ V_i \}_{i = 1}^n$ and $\{ \ve_i \}_{i = 1}^n$ are independent.
  Then, for any $k \in \mathbb N$, it holds that \\
  (a) $ \mathbb P_n (2 \pi(Z'\hat \gamma^*) - 1) D H^1 = o_p (n^{-3/4})$, \\
  (b) $ \mathbb P_n V (\pi (Z'\hat \gamma^*) - 1/2)^k \rho (\ve) = o_p (n^{-1/4})$, \\
  (c) $ \mathbb P_n D^2 (\pi (Z'\hat \gamma^*) - 1/2)^2 \rho (\ve) = o_p(n^{-1/2})$.
\end{lemma}

\begin{proof}[Proof of Lemma \ref{shrinkage pi}]
  \underline{$(a)$}.
  By Proposition \ref{consistency}, there exists a sequence $r_n$ converging to zero such that $\mathbb P (n^{1/4} \sqrt{\log d \log n} \| \hat \gamma^* \|_1 \geq r_n) \rightarrow 0$.
  Hence, it suffices to show $ \|  \mathbb P_n (2\pi (Z' \gamma) - 1) D H^1 \|_{\Gamma_n} = o_p (n^{-3/4})$,
  where $\Gamma_n := \left\{ \gamma \in \Gamma: \| \gamma \|_1 \leq n^{-1/4} (\log d \log n)^{-1/2} r_n \right\}$.
  Because $\mathbb E[(2 \pi (Z'\gamma) - 1) D H^1] = 0$, the symmetrization inequality gives that 
  \begin{align} \label{lemma shrinkage pi a symmetrization}
    \mathbb E \left[  \left\| \mathbb P_n (2\pi (Z' \gamma) - 1) D H^1 \right\|_{\Gamma_n}  \right]
    &\lesssim \ \mathbb E \left[ \left\| \mathbb P_n \xi (2 \pi (Z'\gamma) - 1) D H^1 \right\|_{\Gamma_n} \right], \notag \\
    &\lesssim \ \mathbb E \left[ \max_{1 \leq i \leq n} |\ve_i| \mathbb E \left[ \sup_{\gamma \in \Gamma_n} \left| \frac{1}{n} \sum_{i = 1}^n \frac{\xi_i (\pi (Z'\gamma) - 1/2)D_i \ve_i}{\max_{1 \leq i \leq n} |\ve_i|} \right| | D^{(n)}, Z^{(n)}, \ve^{(n)} \right]   \right],
  \end{align}
  where $\xi$ is a Rademacher random variable independent of $(D, Z, \ve)$, and $D^{(n)} := (D_1, \dots, D_n)$, $Z^{(n)} := (Z_1, \dots, Z_n)$ and $\ve^{(n)} := (\ve_1, \dots, \ve_n)$.
  From Assumption \ref{covariate}$(d)$, we may assume $|D| \leq 1$ without the loss of generality.
  Then a function $\varphi_i (t) := \frac{(\pi(t) - 1/2)D_i \ve_i}{ \max_{1 \leq i \leq n} |\ve_i|}$ is contraction with $\varphi_i (0) = 0$.
  It follows from Theorem 4.12 of \cite{ledoux1991probability} that 
  \begin{equation} \label{lemma shrinkage pi a contraction}
    \mathbb E \left[ \sup_{\gamma \in \Gamma_n} \left| n^{-1} \sum_{i = 1}^n \frac{\xi_i (\pi (Z'\gamma) - 1/2)D_i \ve_i}{\max_{1 \leq i \leq n} |\ve_i|} \right| | D^{(n)}, Z^{(n)}, \ve^{(n)} \right] \leq 2 \mathbb E \left[ \left\| \mathbb P_n \xi Z'\gamma \right\|_{\Gamma_n} | D^{(n)}, Z^{(n)}, \ve^{(n)} \right].
  \end{equation}
  Combining \eqref{lemma shrinkage pi a symmetrization} and \eqref{lemma shrinkage pi a contraction}, we obtain
  \begin{equation} \label{lemma shrinkage pi a independence}
    \mathbb E \left[  \left\| \mathbb P_n (2\pi (Z' \gamma) - 1) D H^1 \right\|_{\Gamma_n}  \right] \lesssim \mathbb  E \left[ \max_{1 \leq i \leq n} |\ve_i| \| \mathbb P_n \xi Z'\gamma \|_{\Gamma_n} \right]
    = \mathbb E \left[ \max_{1 \leq i \leq n} |\ve_i|\right] \mathbb E\left[\| \mathbb P_n \xi Z'\gamma \|_{\Gamma_n}\right],
  \end{equation}
  where the equality follows from the independence of $\ve$ and $(\xi, Z)$.
  For the right side of the equality, $\mathbb E [\max_{1 \leq i \leq n} |\ve_i|] \lesssim \sqrt{\log n}$ follows from Lemma 2.2.1 and Lemma 2.2.2 of \cite{van1996weak} in conjunction with sub-Gaussianity of $\ve$.
  Additionally, $\mathbb E\left[\| \mathbb P_n \xi Z'\gamma \|_{\Gamma_n}\right] \lesssim \sqrt{\frac{\log d}{n}} n^{-1/4} (\log d \log n)^{-1/2} r_n$ from a similar argument to the proof of $(a)$ in Lemma \ref{maximal}.
  Those two inequality combined with \eqref{lemma shrinkage pi a independence} give that $\mathbb E \left[  \left\| \mathbb P_n (2\pi (Z' \gamma) - 1) D H^1 \right\|_{\Gamma_n}  \right] \lesssim r_n n^{-3/4} = o(n^{-3/4})$.
  We now complete the proof by applying Markov's inequality.

  \underline{$(b)$}.
  Similarly to $(a)$, it suffices to show that $\|\mathbb P_n V (\pi (Z'\gamma) - 1/2)^k \rho (\ve) \|_{\Gamma_n} = o_p (n^{-1/4})$.
  By the mean value theorem and $|\pi (z) - 1/2| \leq 1$ for any $z \in \mathbb R$, 
  \begin{align} \label{lemma shrinkage pi b mean value}
    \mathbb E \left[ \| \mathbb P_n V (\pi (Z'\gamma) - 1/2)^k \rho (\ve) \|_{\Gamma_n} \right] \lesssim \ &\mathbb E \left[  \left\| \mathbb P_n |V| |Z'\gamma| |\rho(\ve)|  \right\|_{\Gamma_n} \right] \notag \\
    \leq \ & \mathbb E \left[ |V| \sup_{\gamma \in \Gamma_n} |Z'\gamma| |\rho (\ve)| \right] \notag \\
    \leq \ & \mathbb E \left[ |V| \max_{1 \leq j \leq d} |Z_{(j)}| \right] \mathbb E[|\rho (\ve)|] \sup_{\gamma \in \Gamma_n} \| \gamma \|_1 ,
  \end{align}
  where the last inequality follows from the independence of $V$ and $\ve$. By the Cauchy-Schwarz inequality, $\mathbb E \left[ |V| \max_{1 \leq j \leq d} |Z_{(j)}| \right] \leq (\mathbb E [|V|^2])^{1/2} (\mathbb E[\max_{1 \leq j \leq d} |Z_{(j)}|^2])^{1/2}
  \lesssim \sqrt{\log d}$, where the last inequality follows from the assumption on the moment of $V$ and the argument leading to \eqref{lemma maximal square bound}.
  Combining this inequality with \eqref{lemma shrinkage pi b mean value}, the finiteness of the moment $\mathbb E[|\rho (\ve)|]$ and $\sup_{\gamma \in \Gamma_n} \| \gamma \|_1 \leq n^{-1/4} (\log d \log n)^{-1/2} r_n$ yields 
  $\mathbb E \left[ \| \mathbb P_n V (\pi (Z'\gamma) - 1/2)^k \rho (\ve) \|_{\Gamma_n} \right]   \lesssim r_n n^{-1/4} / \sqrt{\log n} = o(n^{-1/4}).$
  Applying Markov's inequality completes the proof.

  \underline{$(c)$}.
  Similarly to $(a)$, it suffices to show that $\| \mathbb P_n D^2 (\pi (Z'\gamma) - 1/2)^2 \rho (\ve) \|_{\Gamma_n} = o_p (n^{-1/2})$.
  By the mean value theorem in conjunction with Assumption \ref{covariate}$(d)$,
  \begin{equation} 
    \mathbb E\left[ \| \mathbb P_n D^2 (\pi (Z'\gamma) - 1/2)^2 \rho (\ve) \|_{\Gamma_n} \right] \lesssim \mathbb E \left[ \sup_{\gamma \in \Gamma_n} |Z'\gamma|^2 |\rho (\ve)| \right] 
    \lesssim  \mathbb E \left[ \max_{1 \leq j \leq d} Z_{(j)}^2 \right]  \sup_{\gamma \in \Gamma_n} \| \gamma \|_1^2, \notag
  \end{equation}
  where the second inequality follows from the independence of $\ve$ and $Z$, and the finiteness of the moment $\mathbb E[|\rho (\ve)|]$.
  For the right side, a similar argument to the proof of $(b)$ gives that $\mathbb E \left[ \max_{1 \leq j \leq d} Z_{(j)}^2 \right]  \lesssim \log d$.
  Additionally, we have $\sup_{\gamma \in \Gamma_n} \| \gamma \|_1^2 \leq n^{-1/2} (\log d \log n)^{-1} r_n^2$ from the choice of $\Gamma_n$.
  Therefore, we arrive at $\mathbb E\left[ \| \mathbb P_n D^2 (\pi (Z'\gamma) - 1/2)^2 \rho (\ve) \|_{\Gamma_n} \right] \lesssim r_n^2 n^{-1/2} / \log n = o(n^{-1/2})$. We complete the proof by applying Markov's inequality.
\end{proof}

\bibliographystyle{asa}
\bibliography{shrinkage_subgroup}

\end{document}